\documentclass[10pt]{article}

\usepackage{amssymb,amsthm,amsmath,hyperref}

\usepackage{graphicx,float}

\allowdisplaybreaks

\numberwithin{equation}{section}
\newtheorem{thm}{Theorem}[section]
\newtheorem{lem}[thm]{Lemma}
\newtheorem{prop}[thm]{Proposition}
\theoremstyle{remark}
\newtheorem{rmk}{Remark}[section]
\theoremstyle{definition}
\newtheorem{defn}{Definition}[section]

\begin{document}

\title{Zero-electron-mass limit of Euler-Poisson equations}

\author{{Jiang Xu}$^1$\thanks{jiangxu\underline{
}79@yahoo.com.cn} , Ting Zhang$^2$\thanks{zhangting79@zju.edu.cn}
  \\ \textit{\small 1.  Department of Mathematics, Nanjing University of Aeronautics and Astronautics,
              Nanjing 211106, P.R.China}  \\
     \textit{\small 2. Department of Mathematics, Zhejiang University,
Hangzhou 310027, China}}
\date{}

\maketitle

\begin{abstract}
This paper is concerned with multidimensional Euler-Poisson
equations for plasmas. The equations take the form of Euler
equations for the conservation laws of the mass density and current
density for charge-carriers (electrons and ions), coupled to a
Poisson equation for the electrostatic potential. We study the limit
to zero of some physical parameters which arise in the scaled
Euler-Poisson equations, more precisely, which is the limit of
vanishing ratio of the electron mass to the ion mass. When the
initial data are small in critical Besov spaces, by virtue of the
``Shizuta-Kawashima" skew-symmetry condition, we establish the
uniform global existence and uniqueness of classical solutions. Then
we develop new frequency-localization Strichartz-type estimates for
the equation of acoustics (a modified wave equation) with the aid of
the detailed analysis of the semigroup formulation generated by this
modified wave operator. Finally, it is shown that the uniform
classical solutions converge towards that of the incompressible
Euler equations (for \textit{ill-prepared} initial data) in a
refined way as the scaled electron-mass tends to zero.
\end{abstract}

\section{Introduction}
The Euler-Poisson equations, sometimes called as the (macroscopic)
hydrodynamic models, are to describe the transport of
charge-carriers (electrons and ions) in semiconductor devices or
plasmas. The system consists of the conservation laws for the mass
density and current density for carriers, with a Poisson equation
for the electrostatic potential. Using the classical moment method,
it is derived from the semi-classical Boltzmann-Poisson equations,
for details, see \cite{MRS}. In recent three decades, the
Euler-Poisson equations has attracted increasing attention both
numerical simulation and theoretical analysis, since it can capture
more physical phenomena that occur in submicron devices or in the
presence of high electric fields than the traditional
drift-diffusion models.

More specifically, we consider an un-magnetized plasmas consisting
of electrons with charge $q_e = -1$ and of a single species of ions
with charge $q_i = +1$. Denote by $n_e = n_e(t,x), \mathbf{v}_e =
\mathbf{v}_e(t,x)$ (resp., $n_i, \mathbf{v}_i$) the scaled density
and mean velocity of the electrons (resp., ions) and by
$\phi=\phi(t,x)$ the scaled electrostatic potential. From
\cite{JP1}, these unknowns satisfy the following scaled
Euler-Poisson equations:
\begin{equation}
\left\{
\begin{array}{l}
\frac{\partial}{\partial t}n_a + \mathrm{div}(n_a\mathbf{v}_a) = 0,
\\[4mm]
\delta_a\frac{\partial}{\partial t}(n_a\mathbf{v}_a) +
\delta_a\mathrm{div}(n_a\mathbf{v}_a\otimes\mathbf{v}_a) + \nabla
P_a(n_{a})= - q_an_a\nabla\phi - \delta_a\frac{n_a\mathbf{v}_a}{\tau_a},\\[4mm]
\lambda^2\Delta\phi = n_e - n_i + C(x),
\end{array} \right.\label{M-E1}
\end{equation}
where  $a = e, i$ and $(t,x)\in [0,+\infty)\times
\mathbb{R}^{N}(N\geq2)$. The symbols $\mathrm{div},\nabla,\Delta$
and $\otimes$ are the  $x$-divergence operator, gradient
operator, Laplacian operator  and the tensor product of two vectors respectively;
the parameters $\lambda,\tau_{a}>0$ are the scaled constants for the
Debye length and the momentum relaxation time of electrons (if $a =
e$; otherwise, ions if $a = i$) respectively; the pressure
$P_{a}(n_{a})$ is a smooth function satisfying
\begin{eqnarray*}
P'_{a}(n_{a})>0 \ \ \mbox{for all}\ \ n_{a}>0.
\end{eqnarray*}
For the sake of (mathematical) simplicity, we assume that it
satisfies the usual ``$\gamma$-law"
\begin{eqnarray*}
P_{a}(n_{a})=c_{a}n_{a}^{\gamma_{a}},
\end{eqnarray*}
where $c_{a}>0$ is a physical constant. In such case, the plasmas
is called \emph{isothermal} if $\gamma_{a}=1$ and \emph{isentropic}
if $\gamma_{a}>1$. The function $C(x)$, which only depends on the
space variable, represents the density of fixed charged background
ions (doping profile).

The dimensionless parameter $\delta_a$ in the Euler-Poisson
equations (\ref{M-E1}) is given by
$$
\delta_a = \frac{m_a v^2_0}{k_BT_0},
$$
where $m_e$ (resp. $m_i$) is the mass of a single electron (resp.
ion), $k_B$ is the Boltzmann constant, $v_0$ and $T_0$ are typical
velocity and temperature values for the plasmas respectively. The reader is refer
to \cite{JP1} for details about the scaling and the physical
assumptions. Since $m_i$ is not smaller than the mass of a proton,
we have
$$
\frac{m_e}{m_i}\leq \frac{m_e}{\mbox{the mass of a proton}} \simeq
5.45\times 10^{- 4}.
$$
Namely, $\frac{m_e}{m_i}$ is a small number. Thus, if $v^2_0$ is
chosen to be $k_BT_0/m_i$, then $\delta_i = 1$ and $\delta_e =
\frac{m_e}{m_i} \ll1$. The main goal of this paper is to investigate
the limit as $\delta_e$ goes to zero.

In plasmas physics, the zero-electron-mass assumption (\textit{i.e.}
$\delta_e = 0$) is widely used, \textit{e.g.}, see \cite{He,KD}. For
simplicity, we assume that $\bar{n}: = n_i - C(x)$ is a given
positive constant and therefore consider the unipolar Euler-Poisson
equations only, instead of (\ref{M-E1}):
\begin{equation}
\left\{
\begin{array}{l}
n_t + \mathrm{div}(n\mathbf{v}) = 0 \\
(n\mathbf{v})_t + \mathrm{div}(n\mathbf{v}\otimes\mathbf{v}) +
\frac{\nabla P(n)}{\epsilon^2} =
\frac{n\nabla\phi}{\epsilon^2}-n\mathbf{v}\\
\Delta\phi = n- \bar{n},
\end{array} \right.\label{M-E2}
\end{equation}
with the initial data
\begin{equation}
(n,\mathbf{v})(x,0)=(n_{0},\mathbf{v}_{0}), \label{M-E3}
\end{equation}
where we set $\tau_e=1=\lambda$ and
$$
n = n_e,\qquad \mathbf{v} = \mathbf{v}_e, \qquad
\delta_e=\epsilon^2, \qquad P = P_e=An^\gamma.
$$

At the formal level, if $\nabla P(n)\rightarrow0,
\nabla\phi\rightarrow0,$ when $\epsilon$ goes to zero. Hence $n$
must be the positive constant $\bar{n}$. Then passing to the limit
in the mass conservation equation, we get
$\mathrm{div}\mathbf{v}\rightarrow0$. Coming back to the momentum
equation, we conclude that $\mathbf{u}$ (the limit of $\mathbf{v}$)
must satisfy the incompressible Euler equations with damping
\begin{equation*}
      \left\{
      \begin{array}{l}
    \mathbf{u}_t+\mathbf{u}\cdot\nabla \mathbf{u}+\mathbf{u}+\nabla \Pi=0,\\
    \mathrm{div}\mathbf{u}=0,\\
        \mathbf{u}|_{t=0}=u_0,
      \end{array}
      \right.
    \end{equation*}
which also can be written
\begin{equation}
      \left\{
      \begin{array}{l}
    \mathbf{u}_t+\mathcal{P}(\mathbf{u}\cdot\nabla \mathbf{u})+\mathbf{u}=0,\\
    \mathrm{div}\mathbf{u}=0,\\
        \mathbf{u}|_{t=0}=u_0=\mathcal{P}\mathbf{v}_0,
      \end{array}
      \right.\label{M-E4}
    \end{equation}
where $\mathcal{P}$ stands for the Leray projector on solenoidal
vector fields.

Moreover, if we assume that $n=\bar{n}+O(\epsilon^2),
\mathrm{div}\mathbf{v}=O(\epsilon)$, then $\nabla
P(n)=O(\epsilon^2), \nabla\phi=O(\epsilon^2)$. Indeed, this entails
that $n_{t}, \mathbf{v}_{t}$ are uniformly bounded so that the
dangerous time oscillations can not occur. Starting from this simple
consideration, Al\`{\i}, Chen, J\"{u}ngel and Peng \cite{ACJP}
justified the zero-electron-mass limit ($\epsilon\rightarrow0$) of
(\ref{M-E2}) with data of the following type
$$
n_{0}=\bar{n}+\epsilon^2n_{0,1},\ \ \ \
\mathbf{v}_{0}=\bar{\mathbf{v}}(x,0)+\epsilon\mathbf{v}_{0,1}
$$
with $\mathrm{div}\bar{\mathbf{v}}(x,0)=0$, and $(n_{0,1},
\mathbf{v}_{0,1})$ uniformly bounded in the the periodic Sobolev
space $H^{\sigma}(\mathbb{T}^N)(\sigma>1+N/2)$. Such data are
referred as \textit{well-prepared} data in the usual PDE's
terminology. In the case of \textit{well-prepared} data, the
zero-electron-mass limit has some similarities with the classical
low-Mach-number limit in the compressible Euler equations studied by
Klainerman and Majda \cite{KM1,KM2}. The unique difference is that
there is an extra singularity from the electron-field term. Al\`{\i}
\textit{et al.} \cite{ACJP} overcame this difficulty by a careful
use of the mass conservation and the Poisson equation, and obtained
the uniform \textit{a priori} estimates with respect to the scaled
zero-electron-mass. Then they justified the zero-electron-mass limit
for smooth solutions from (\ref{M-E2})-(\ref{M-E3}) to (\ref{M-E4})
(incompressible limit) by virtue of the Aubin-Lions compactness
lemma. The precise limit behaviors of smooth solutions were shown by
the recent work \cite{XY2}. The zero-electron-mass limit for weak
entropy solutions was studied by Goudon, J\"{u}ngel and Peng
\cite{GJP} under some restrictive assumptions, with the help of the
kinetic formulation, the monotonicity and compactness arguments. It
is worth noting that these results only provided \textit{local}
convergence or the case of \textit{well-prepared} data.

In the present paper, we shall study the zero-electron-mass limit
since few works can be found in mathematics literatures, except
\cite{ACJP,GJP}. We focus on the case of \textit{ill-prepared} data
and hope to get the \textit{global} convergence of classical
solutions. We make a weaker assumption that the initial density only
satisfies $n_{0}=\bar{n}+\epsilon n_{0,1}$ with
$(n_{0,1},\mathbf{v}_{0})$ uniformly bounded (in an appropriate
functional space).

\subsection{Symmetry}
It is convenient to state the basic ideas and main results of this
paper, we first introduce a function transform to reduce
(\ref{M-E2}) to a symmetric hyperbolic-elliptic form.

For the isentropic case $(\gamma>1)$, by defining the sound speed
$$\psi(n)=\sqrt{p'(n)},$$
and the sound speed $\bar{\psi}=\psi(\bar{n})$ at a background
density $\bar{n}$, we set
\begin{equation}m=\frac{2}{\gamma-1}\Big(\psi(n)-\bar{\psi}\Big).\label{M-E5}\end{equation}
Then (\ref{M-E2}) can be rewritten as
\begin{equation}
\left\{
\begin{array}{l}m_{t}+\bar{\psi}\mbox{div}\mathbf{v} =
-\mathbf{v}\cdot\nabla m-\frac{\gamma-1}{2}m\mbox{div}\mathbf{v},\cr
 \mathbf{v}_{t}+\bar{\psi}\epsilon^{-2}\nabla m =
 -\mathbf{v}\cdot\nabla\mathbf{v}-\frac{\gamma-1}{2}\epsilon^{-2}m\nabla m+
 \epsilon^{-2}\nabla\phi - \mathbf{v},\cr
 \Delta\phi=h(m),
 \end{array} \right.\label{M-E6}
\end{equation}
where
$$
h(m) =
\left\{(A\gamma)^{-\frac{1}{2}}(\frac{\gamma-1}{2}m+\bar{\psi})
\right\}^{\frac{2}{\gamma-1}}-\bar{n}
$$
is a smooth function on the domain
$\{m|\frac{\gamma-1}{2}m+\bar{\psi}>0\}$ satisfying $h(0)=0$. The
corresponding initial data become
\begin{equation}
(m_{0},\mathbf{v}_{0},\nabla\phi_{0}) =
\Big\{\frac{2}{\gamma-1}\Big(\psi(n_{0})-\bar{\psi}\Big),\mathbf{v}_{0},
\nabla\Delta^{-1}(n_{0}-\bar{n}) \Big\} . \label{M-E7}
\end{equation}

For the isothermal case where $\gamma=1$, the form (\ref{M-E6}) is
still valid with $\bar\psi = \sqrt{A}$. However, it depends on the
following enthalpy variable change
\begin{eqnarray}
m = \sqrt{A}(\ln n - \ln\bar n)\label{M-E8},
\end{eqnarray}
which the details are referred to \cite{FXZ}. About the equivalence
for classical solutions away from the vacuum between
(\ref{M-E2})-(\ref{M-E3}) and (\ref{M-E6})-(\ref{M-E7}), see Section
3. Moreover we introduce the variables as follows:
$$
m^{\epsilon}=\frac{m}{\epsilon},\ \ \
\mathbf{v}^{\epsilon}=\mathbf{v},\ \ \
\nabla\phi^{\epsilon}=\frac{\nabla\phi}{\epsilon}.
$$
Then the new variables satisfy
\begin{equation}
\left\{
\begin{array}{l}m_{t}^{\epsilon} + \bar{\psi}\epsilon^{-1}\mbox{div}\mathbf{v}^{\epsilon}
=-\mathbf{v}^{\epsilon}\cdot\nabla m^{\epsilon}-
\frac{\gamma-1}{2}m^{\epsilon}\mbox{div}\mathbf{v}^{\epsilon},\cr
 \mathbf{v}^{\epsilon}_{t}+\bar{\psi}\epsilon^{-1}\nabla m^{\epsilon}=
 -\mathbf{v}^{\epsilon}\cdot\nabla\mathbf{v}^{\epsilon}-
 \frac{\gamma-1}{2}m^{\epsilon}\nabla m^{\epsilon}
 +\epsilon^{-1}\nabla\phi^{\epsilon} - \mathbf{v}^{\epsilon},\cr
\Delta\phi^{\epsilon}=\epsilon^{-1}h(\epsilon m^{\epsilon}),
\end{array} \right.\label{M-E9}
\end{equation}
with the initial data
\begin{equation}
(m^{\epsilon}_{0},\mathbf{v}^{\epsilon}_{0},\nabla\phi^{\epsilon}_{0})
= \Big\{\frac{2}{\gamma-1}\Big(\frac{\psi(
n_{0})-\bar{\psi}}{\epsilon}\Big),\mathbf{v}_{0},
\nabla\Delta^{-1}\Big(\frac{n_{0}-\bar{n}}{\epsilon}\Big) \Big\}.
\label{M-E10}
\end{equation}

One expects $\mathbf{v}^{\epsilon}$ to tend to $\mathbf{u}$ in
(\ref{M-E9}) which $\mathbf{u}$ solves the incompressible Euler
equations (\ref{M-E4}) as $\epsilon\rightarrow0$. The expected
convergence however is not easy to justify rigorously. Here, there
are two main difficulties. The first is the singularity from the
electron-field term, which can not be overcome by using the
symmetrizer of main part of hyperbolic system as in \cite{KM1,KM2}.
We adopt a new (but small!) technique to deal with the singular
electron-field term, which can help us to simplify the similar
analysis as \cite{ACJP} heavily, see (\ref{M-E19})-(\ref{M-E20}).
The second is that one has to face the propagation of acoustic waves
with the speed $\epsilon^{-1}$, a phenomenon which does not occur in
the case of \textit{well-prepared} data. To solve this, we need to
develop some new ideas. Inspired by \cite{Danchin}, we split the
velocity into a divergence-free part and a gradient part to obtain
the equation of acoustics
\begin{equation*}\left\{\begin{array}{l}
      m^\epsilon_t+\bar{\psi}\frac{\Lambda d^\epsilon}{\epsilon}=F,\\
    d^\epsilon_t+d^\epsilon-\bar{\psi}\frac{\Lambda
    m^\epsilon}{\epsilon}-\frac{h'(0)\Lambda^{-1}m^\epsilon}{\epsilon}
    = G,
     \end{array}
    \right.
        \end{equation*}
i.e.,
    \begin{eqnarray*}
    d^\epsilon_{tt}+d^\epsilon_t-\frac{\bar{\psi}^2}{\epsilon^2}\Delta d^\epsilon +\frac{h'(0)\bar{\psi}}{\epsilon^2}d^\epsilon
    = G_t+\frac{\bar{\psi}\Lambda}{\epsilon}F+\frac{h'(0)\Lambda^{-1}}{\epsilon}F,
        \end{eqnarray*}
(for details, see Section 5). Then we establish some dispersive
estimates according to the semigroup theory of this modified wave
operator, and achieve a new frequency-localization Strichartz-type
estimate which is used to pass to the zero-electron-mass limit in a
refined way.

Based on the recent work \cite{FXZ,I}, we still choose the critical
Besov space framework in space-variable $x$ (a subalgebra of
${\mathcal{W}}^{1,\infty}$) to study the global well-posedness and
the zero-electron-mass limit of classical solutions to the system
(\ref{M-E2})-(\ref{M-E3}). The main results are states as follows.

\subsection{Main results}
\begin{thm}\label{thm1.1}
Set $\sigma=1+N/2$. There is a positive constant $\delta_{0}$
independent of $\epsilon$, such that if
$$\Big\|\Big(\frac{n_{0}-\bar{n}}{\epsilon},\mathbf{v}_{0},\frac{\mathbf{e}_{0}}{\epsilon}\Big)\Big\|_{B^{\sigma}_{2,1}(\mathbb{R}^{N})}\leq \delta_{0}$$
for $0<\epsilon\leq\epsilon_{0}$, and $\mathbf{e}_{0}:=
\nabla\Delta^{-1}(n_{0}-\bar{n})$, then the system
(\ref{M-E2})-(\ref{M-E3}) admits a unique global solution
$(n,\mathbf{v},\nabla\phi)$ satisfying
$$
(n-\bar{n},\mathbf{v},\nabla\phi) \in
\mathcal{C}(\mathbb{R}^{+},B^{\sigma}_{2,1}(\mathbb{R}^{N})).
$$
Moreover, the uniform energy estimate holds:
\begin{eqnarray*}&&\Big\|\Big(\frac{n-\bar{n}}{\epsilon},\mathbf{v},
\frac{\nabla\phi}{\epsilon}\Big)(\cdot,t)\Big\|_{B^{\sigma}_{2,1}(\mathbb{R}^{N})}\\
&\leq&
C_{0}\Big\|\Big(\frac{n_{0}-\bar{n}}{\epsilon},\mathbf{v}_{0},\frac{\mathbf{e}_{0}}{\epsilon}\Big)\Big\|_{B^{\sigma}_{2,1}(\mathbb{R}^{N})}\exp(-\mu_{0}t),
\ \ t\geq0,\end{eqnarray*} where $\mu_{0}$, $C_{0}$ are some
positive constants independent of $\epsilon$.
\end{thm}

\begin{rmk}
 The symbol $\nabla\Delta^{-1}$ means
 $$
 \nabla\Delta^{-1}f=\int_{\mathbb{R}^d}\nabla_{x}G(x-y)f(y)dy,
 $$
 where $G(x,y)$ is a solution to $\Delta_xG(x,y)=\delta(x-y)$ with
 $x,y\in\mathbb{R}^N$. In the periodic setting, the regularity assumption
 on $\textbf{e}_{0}$ can be removed.
\end{rmk}

\begin{rmk}
In the proof of Theorem \ref{thm1.1}, different from that in
\cite{FXZ}, ``Shizuta-Kawashima" skew-symmetry condition developed
for general hyperbolic systems of balance laws \cite{KY,Y2} is used,
which helps us avoid differentiating the system with respect to
time-variable $t$ and the proof is shortened. Thanks to the
isentropic Euler-Poisson equations also including isentropic Euler
equations \cite{CG}, the concrete information of skew-symmetry
matrix $K(\xi)$ is well known (unknown for general systems) which is
very effective to estimate the coupled electron-field
$\nabla\phi^{\epsilon}$, see (\ref{M-E34}).
\end{rmk}
\begin{thm}Let the assumptions of Theorem \ref{thm1.1} be fulfilled.
Let $(m^{\epsilon},\mathbf{v}^{\epsilon},\\ \nabla\phi^{\epsilon})$
denote the global solution to (\ref{M-E9})-(\ref{M-E10}), then
$(m^{\epsilon},\mathcal{Q}\mathbf{v}^{\epsilon},\nabla\phi^{\epsilon})$
tends to zero in
$L^{1}(\mathbb{R}^{+};B^{N/p}_{p,1}(\mathbb{R}^{N}))$ and
$\mathcal{P}\mathbf{v}^{\epsilon}$ converges in
$L^{\infty}(\mathbb{R}^{+};B^{N/p}_{p,1}(\mathbb{R}^{N}))\cap
L^{1}(\mathbb{R}^{+};B^{N/p}_{p,1}(\mathbb{R}^{N}))$ towards the
solution $\mathbf{u}\in (L^{\infty}\cap
L^{1})(\mathbb{R}^{+};B^{\sigma}_{2,1}(\mathbb{R}^{N}))$ to the
incompressible Euler equations (\ref{M-E4}), as
$\epsilon\rightarrow0$. Here $2\leq p\leq\infty$, $\mathcal{P}$
stands for the Leray projector on solenoidal vector fields and is
defined by $\mathcal{P}=I-\mathcal{Q}$ with
$\mathcal{Q}=\nabla\Delta^{-1}\mathrm{div}$.
\end{thm}
\begin{rmk}
The general convergence statements are give by Proposition
\ref{P-5.2} and Proposition \ref{P-5.3}, where the speed of
convergence may be characterized in terms of power of $\epsilon$.
\end{rmk}
\begin{rmk}
Let us mention that the approach developed by this paper can also be
applied to study the low-Mach-number limits of the compressible
Euler equations with damping for the perfect gas flow. Therefore,
this work can be regarded as a supplement to the theory of
asymptotic limits for hyperbolic problems.
\end{rmk}
\begin{rmk}
Prescribing \textit{ill-prepared} initial data and periodic boundary
conditions preclude from using dispersive properties in order to
pass to the limit in (\ref{M-E9}). Indeed, there is no chance that
the acoustic waves go at infinity. Hence, there could be resonances
which may hinder the convergence to the incompressible Euler
equations. This will be shown in a forthcoming paper.
\end{rmk}

Our paper is organized as follows. In Section 2, we introduce Besov
spaces and their properties. In Section 3, we give some remarks on
the hyperbolic symmetrization and recall a local existence result of
classical solutions. In Section 4, we derive the \textit{a-priori}
estimate which is used to achieve the global existence of uniform
classical solutions. For clarity, Section 5 is divided into two
parts. Based on the detailed analysis of the equation of acoustics,
we first establish a new frequency-localization Strichartz-type
estimate. Then using this estimate, we perform the
zero-electron-mass limit in a refined way.

Finally, some efforts on other kinds of asymptotic limits (such as
relaxation-time limit and quasineutral limit) of classical solutions
to the Euler-Poisson equations (\ref{M-E2})-(\ref{M-E3}) should be
mentioned, the interested reader is referred to \cite{W,XY,Y1} and
the literature quoted therein.

\section{Preliminary}
Throughout this paper, $C$ is a uniform positive constant
independent of $\epsilon$. $f\approx g$ means that $f\leq Cg$ and
$g\leq Cf$. We denote by $L^{\rho}(0,T;X)$, $\mathcal{C}([0,T],X)$
(resp., $\mathcal{C}^{1}([0,T],X)$) the space of $\rho$-power
integrable and continuation (resp., continuously differentiable)
functions on $[0,T]$ with values in a Banach space $X$,
respectively. In the case $T=+\infty$, we sometimes label
$L^{\rho}(\mathbb{R}^{+};X)$, $\mathcal{C}(\mathbb{R}^{+}, X)$, etc.
For brevity, we use the notation
$\|(a,b,c)\|_{X}:=\|a\|_{X}+\|b\|_{X}+\|c\|_{X}$, where $a,b,c\in
X$. All functional spaces of the present paper are considered in
$\mathbb{R}^{N}$, so we may omit the space dependence for
simplicity.

In this section, we review briefly the Littlewood--Paley
decomposition theory and the characterization of Besov spaces; see
also, \textit{e.g.},~\cite{Danchin} or \cite{FXZ}.

Let $\mathcal{S}$ be the Schwarz class. ($\varphi, \chi)$ is a
couple of smooth functions valued in [0,1] such that $\varphi$ is
supported in the shell
$\textbf{C}(0,\frac{3}{4},\frac{8}{3})=\{\xi\in\mathbb{R}^{N}|\frac{3}{4}\leq|\xi|\leq\frac{8}{3}\}$,
$\chi$ is supported in the ball $\textbf{B}(0,\frac{4}{3})=
\{\xi\in\mathbb{R}^{N}||\xi|\leq\frac{4}{3}\}$ satisfying
$$
\chi(\xi)+\sum_{q\in\mathbb{N}}\varphi(2^{-q}\xi)=1,\ \ \ \ q\in
\mathbb{N},\ \ \xi\in\mathbb{R}^{N}
$$
and
$$
\sum_{k\in\mathbb{Z}}\varphi(2^{-k}\xi)=1,\ \ \ \ k\in \mathbb{Z},\
\ \xi\in\mathbb{R}^{N}\setminus\{0\}.
$$
For $f\in\mathcal{S'}$(denote the set of temperate distributions
which is the dual of $\mathcal{S}$), one can define the Fourier
dyadic blocks as follows:
$$
\Delta_{-1}f:=\chi(D)f=\mathcal{F}^{-1}(\chi(\xi)\mathcal{F}f),\ \ \
\Delta_{q}f:=0 \ \  \mbox{for}\ \  q\leq-2;
$$
$$
\Delta_{q}f:=\varphi(2^{-q}D)f=\mathcal{F}^{-1}(\varphi(2^{-q}|\xi|)\mathcal{F}f)\
\  \mbox{for}\ \  q\geq0;
$$
$$
\dot{\Delta}_{k}f:=\varphi(2^{-k}D)f=\mathcal{F}^{-1}(\varphi(2^{-k}|\xi|)\mathcal{F}f)\
\  \mbox{for}\ \  k\in\mathbb{Z},
$$
where $\mathcal{F}f$, $\mathcal{F}^{-1}f$ represent the Fourier
transform and the inverse Fourier transform on $f$, respectively.
The nonhomogeneous Littlewood--Paley decomposition is
$$
f=\sum_{q \geq-1}\Delta_{q}f \ \ \ \mbox{in}\ \ \ \mathcal{S'}.
$$
Define the low-frequency cut-off by
$$
S_{q}f:=\sum_{p\leq q-1}\Delta_{p}f.
$$
The above Littlewood--Paley decomposition is almost orthogonal in
$L^2$.

\begin{prop}\label{prop2.1}
For any $f, g\in\mathcal{S'}$, the following properties hold:
$$
\Delta_{p}\Delta_{q}f\equiv 0 \ \ \ \mbox{if}\ \ \ |p-q|\geq 2,
$$
$$
\Delta_{q}(S_{p-1}f\Delta_{p}g)\equiv 0\ \ \ \mbox{if}\ \ \
|p-q|\geq 5.
$$
\end{prop}
\unskip

Based on the above Littlewood--Paley decomposition, we introduce the
explicit definitions of nonhomogeneous Besov spaces.

\begin{defn}\label{defn2.1}
Let $1\leq p\leq\infty$ and $s\in \mathbb{R}$. For $1\leq r<\infty$,
the Besov spaces  $B^{s}_{p,r}$ are defined by
$$
f\in B^{s}_{p,r} \Leftrightarrow
\biggl(\sum_{q\geq-1}(2^{qs}\|\Delta_{q}f\|_{L^{p}})^{r}\biggr)^{\frac{1}{r}}<\infty
$$
and $B^{s}_{p,\infty}$ are defined by
$$
f\in B^{s}_{p,\infty} \Leftrightarrow
\sup_{q\geq-1}2^{qs}\|\Delta_{q}f\|_{L^{p}}<\infty.
$$
\end{defn}
Some conclusions as follows will be used in subsequent analysis. The
first one is the classical Bernstein's inequality.

\begin{lem}[Bernstein's inequality]\label{lem2.1}
Let $k\in\mathbf{N}$ and $0<R_{1}<R_{2}$. There exists a constant
$C$ depending only on $R_{1},R_{2}$, and $N$ such that for all
$1\leq a\leq b\leq\infty$ and $f\in L^{a}$, we have
$$
\mathrm{Supp}\ \mathcal{ F}f\subset
\textbf{B}(0,R_{1}\lambda)\Rightarrow\sup_{|\alpha|=k}\|\partial^{\alpha}f\|_{L^{b}}\leq
C^{k+1}\lambda^{k+N(\frac{1}{a}-\frac{1}{b})}\|f\|_{L^{a}},
$$
$$
\mathrm{Supp}\ \mathcal{ F}f\subset
\textbf{C}(0,R_{1}\lambda,R_{2}\lambda)\Rightarrow
C^{-k-1}\lambda^{k}\|f\|_{L^{a}}\leq
\sup_{|\alpha|=k}\|\partial^{\alpha}f\|_{L^{a}}\leq
C^{k+1}\lambda^{k}\|f\|_{L^{a}},
$$
where $\mathcal{F}f$(or
$\widehat{f}=\int_{\mathbb{R}^{N}}f(x)\exp(-ix\cdot\xi)dx$)
represents the Fourier transform on $f$.
\end{lem}

The second one is the embedding properties in Besov spaces.
\begin{lem}\label{lem2.2}
$$B^{s}_{p,r}\hookrightarrow B^{\tilde{s}}_{p,\tilde{r}}\ \ \
\mbox{whenever}\ \ \tilde{s}<s\ \ \mbox{or}\ \ \tilde{s}=s \ \
\mbox{and}\ \ r\leq\tilde{r};$$
$$B^{s}_{p,r}\hookrightarrow B^{s-N(\frac{1}{p}-\frac{1}{\tilde{p}})}_{\tilde{p},r}\ \ \
\mbox{whenever}\ \ \tilde{p}>p;$$
$$B^{d/p}_{p,1}(1\leq p<\infty)\hookrightarrow\mathcal{C}_{0},\ \ \ B^{0}_{\infty,1}\hookrightarrow\mathcal{C}\cap L^{\infty},$$
where $\mathcal{C}_{0}$ is the space of continuous bounded functions
which decay at infinity.
\end{lem}
The three one is the compactness result for Besov spaces.
\begin{prop}\label{prop2.2}
Let $1\leq p,r\leq \infty,\ s\in \mathbb{R}$, and $\epsilon>0$. For
all $\phi\in C_{c}^{\infty}$, the map $f\mapsto\phi f$ is compact
from $B^{s+\epsilon}_{p,r}$ to $B^{s}_{p,r}$.
\end{prop}

The last one is a continuity result for compositions.

\begin{prop}\label{prop2.3}
Let $1\leq p,r\leq \infty$,\ and $I$ be an open interval of
$\mathbb{R}$. Let $s>0$ and let $n$ be the smallest integer such
that $n\geq s$. Let $F:I\rightarrow\mathbb{R}$ satisfy $F(0)=0$ and
$F'\in W^{n,\infty}(I;\mathbb{R}).$ Assume that $v\in B^{s}_{p,r}$
takes values in $J\subset\subset I$. Then $F(v)\in B^{s}_{p,r}$ and
there exists a constant $C$ depending only on $s,I,J$, and $N$ such
that
$$
\|F(v)\|_{B^{s}_{p,r}}\leq
C(1+\|v\|_{L^{\infty}})^{n}\|F'\|_{W^{n,\infty}(I)}\|v\|_{B^{s}_{p,r}}.
$$
\end{prop}

\section{Local existence}\setcounter{equation}{0}
In order to obtain the effective \textit{a priori} estimates by the
low- and high-frequency decomposition methods, we formulate
(\ref{M-E2}) into the symmetric hyperbolic-elliptic form
(\ref{M-E6}) in the Introduction.

Here, we give some remarks.
\begin{rmk}($\gamma>1$)\label{rem3.1}
(\ref{M-E5}) induces a variable change from the half-space
$\{(n,\mathbf{v},\nabla\phi)\in (0,+\infty)\times
\mathbb{R}^{N}\times \mathbb{R}^{N}\}$ to the open set
$\{(m,\mathbf{v},\nabla\phi)\in \mathbb{R}\times
\mathbb{R}^{N}\times
\mathbb{R}^{N}|\frac{\gamma-1}{2}m+\bar{\psi}>0\}$. It is easy to
show that for classical solutions $(n, \mathbf{v},\nabla\phi)$ away
from the vacuum, the system (\ref{M-E2}) is equivalent to
(\ref{M-E6}).
\end{rmk}

For the isothermal case ($\gamma=1$), we also have the similar
equivalence. Note that $\sqrt{A}\ln n$ is the enthalpy, instead of
the sound speed.
\begin{rmk}\label{rem3.2}
(\ref{M-E8}) induces a variable change from the half-space
$\{(n,\mathbf{v},\nabla\phi)\in (0,+\infty)\times
\mathbb{R}^{N}\times \mathbb{R}^{N}\}$ to the whole space
$\{(m,\mathbf{v},\nabla\phi)\in \mathbb{R}\times
\mathbb{R}^{N}\times \mathbb{R}^{N}\}$. It is easy to show that for
classical solutions $(n, \mathbf{v},\nabla\phi)$ away from the
vacuum, the system (\ref{M-E2}) is equivalent to (\ref{M-E6}).
\end{rmk}

The symmetric hyperbolic system (\ref{M-E9}) can also be rewrite as
the vector form
\begin{equation}\partial_{t}W^{\epsilon}+\sum_{j=1}^{N}A^{\epsilon}_{j}(\mathbf{v}^{\epsilon})\partial_{x_{j}}W^{\epsilon}=\left(%
\begin{array}{c}
  -\frac{\gamma-1}{2}m^{\epsilon}\mathrm{div}\, \mathbf{v}^{\epsilon} \\
   -\mathbf{v}^{\epsilon}-\frac{\gamma-1}{2}m^{\epsilon}\nabla m^{\epsilon}+\epsilon^{-1}\nabla\phi^{\epsilon} \\
\end{array}%
\right),\label{M-E11}
\end{equation}
coupled with the dynamic electron-potential equation
\begin{equation}
 \Delta\phi^{\epsilon}=\epsilon^{-1}h(\epsilon m^{\epsilon}),\label{M-E12}
\end{equation}
where
$$ W^{\epsilon}=\left(%
\begin{array}{c}
  m^{\epsilon} \\
  \mathbf{v}^{\epsilon}\\\end{array}%
  \right),\
A^{\epsilon}_{j}(\mathbf{v}^{\epsilon})=\left(%
\begin{array}{cc}
  v^{\epsilon j} & \frac{\bar{\psi}}{\epsilon}e_{j}^\top \\
  \frac{\bar{\psi}}{\epsilon}e_{j} & v^{\epsilon j}I_{N\times N} \\
\end{array}%
\right)$$
$$(I_{N\times N}\ \mbox{denotes the unit matrix of order}\  N $$$$\mbox{and}\ e_{j}\ \mbox{is}\ N\mbox{-dimensional
vertor where the $j$th component is one, others are zero} ).$$

Now, we recall a local existence result on classical solutions to
(\ref{M-E9})-(\ref{M-E10}), which has been obtained in \cite{FXZ}.

\begin{prop}\label{prop3.1} For any fixed $\epsilon\in(0,1]$,
suppose that
$(m^{\epsilon}_{0},\mathbf{v}^{\epsilon}_{0},\nabla\phi^{\epsilon}_{0})\in{B^{\sigma}_{2,1}}$
satisfying $\frac{\gamma-1}{2}\epsilon
m_{0}^{\epsilon}+\bar{\psi}>0$, then there exist a time $T_{0}>0$
and a unique solution $(m^{\epsilon}, \mathbf{v}^{\epsilon},
\nabla\phi^{\epsilon})$ to (\ref{M-E9})-(\ref{M-E10}) such that
$(m^{\epsilon}, \mathbf{v}^{\epsilon}, \nabla\phi^{\epsilon})\in
\mathcal{C}^{1}([0,T_{0}]\times \mathbb{R}^N)$ with
$\frac{\gamma-1}{2}\epsilon m^{\epsilon}+\bar{\psi}>0$ for all
$t\in[0,T_{0}]$ and $(m^{\epsilon}, \mathbf{v}^{\epsilon},
\nabla\phi^{\epsilon})\in
\mathcal{C}([0,T_{0}],B^{\sigma}_{2,1})\cap
\mathcal{C}^1([0,T_{0}],B^{\sigma-1}_{2,1})$.
\end{prop}

\section{A uniform \textit{a priori} estimate}\setcounter{equation}{0}
In this section, we establish a crucial \textit{a priori} estimate
by using the low- and high-frequency decomposition methods, which is
used to derive the global existence and exponential stability of
classical solutions to (\ref{M-E9})-(\ref{M-E10}).
\begin{prop}\label{prop4.1}
There exist three positive constants $\delta_{1}, C_{1}$ and
$\mu_{1}$ independent of $\epsilon$, such that for any $T>0$, if
\begin{eqnarray}\sup_{0\leq t\leq
T}\|(m^{\epsilon},\mathbf{v}^{\epsilon},
\nabla\phi^{\epsilon})(\cdot,t)\|_{B^{\sigma}_{2,1}}\leq
\delta_{1},\label{M-E13}\end{eqnarray} then
\begin{eqnarray}\|(m^{\epsilon},\mathbf{v}^{\epsilon},
\nabla\phi^{\epsilon})(\cdot,t)\|_{B^{\sigma}_{2,1}}\leq
C_{1}\Big\|\Big(\frac{m_{0}}{\epsilon},\mathbf{v}_{0},\frac{\mathbf{e}_{0}}{\epsilon}\Big)\Big\|_{B^{\sigma}_{2,1}}\exp(-\mu_{1}t),
\label{M-E14}
 \end{eqnarray}
where $t\in[0,T]$.
\end{prop}

Having this Proposition \ref{prop4.1}, we can extend the
local-in-time solutions in Proposition \ref{prop3.1} by virtue of
the standard continuation argument and obtain the global existence
of uniform classical solutions to the system
(\ref{M-E9})-(\ref{M-E10}). Using the imbedding property in Besov
space $B^{\sigma}_{2,1}$, we know
$(m^{\epsilon},\mathbf{v}^{\epsilon},\nabla\phi^{\epsilon})\in
\mathcal{C}^{1}([0,\infty)\times \mathbb{R}^{N})$ solves
(\ref{M-E9})-(\ref{M-E10}). The choice of $\delta_{1}$ is sufficient
to ensure $\frac{\gamma-1}{2}\epsilon m^{\epsilon}+\bar{\psi}>0$ for
$0<\epsilon\leq1$. From the Remark \ref{rem3.1}, we achieve that
$(n,\mathbf{v},\nabla\phi)\in \mathcal{C}^{1}([0,\infty)\times
\mathbb{R}^{N})$ is a solution of (\ref{M-E2})-(\ref{M-E3}) with
$n>0$.\ Furthermore, we arrive at Theorem \ref{thm1.1}.

The main ingredients in the proof of Proposition \ref{prop4.1} are
the high-frequency $(q\geq0)$ estimates and low-frequency $(q=-1)$
estimates on $(m^{\epsilon},\mathbf{v}^{\epsilon},
\nabla\phi^{\epsilon})$. We divide it into several lemmas, since the
proof is a bit longer.
\begin{lem}\label{lem4.1}
If $(m^{\epsilon},\mathbf{v}^{\epsilon}, \nabla\phi^{\epsilon})\in
\mathcal{C}([0,T],B^{\sigma}_{2,1})\cap
\mathcal{C}^{1}([0,T],B^{\sigma-1}_{2,1})$ is a solution of
(\ref{M-E9})-(\ref{M-E10}) for any given $T>0$, then the following
estimate holds ($q\geq-1$):
\begin{eqnarray}&&\frac{1}{2}\frac{d}{dt}\Big(\|\Delta_{q}m^{\epsilon}\|^2_{L^2}+\|\Delta_{q}\mathbf{v}^{\epsilon}\|^2_{L^2}+\frac{1}{\bar{n}}\|\Delta_{q}\nabla\phi^{\epsilon}\|^2_{L^2}\Big)
+\|\Delta_{q}\mathbf{v}^{\epsilon}\|^2_{L^2}\nonumber\\
&\leq&
C\|(\Delta_{q}m^{\epsilon},\Delta_{q}\mathbf{v}^{\epsilon},\Delta_{q}\nabla\phi^{\epsilon})\|_{L^2}
\Big\{\|(\nabla
m^{\epsilon},\nabla\mathbf{v}^{\epsilon})\|_{L^{\infty}}\|(\Delta_{q}m^{\epsilon},\Delta_{q}\mathbf{v}^{\epsilon})\|_{L^2}
\nonumber\\&&+\|[\mathbf{v}^{\epsilon},\Delta_{q}]\cdot\nabla
m^{\epsilon}\|_{L^2}+\|[\mathbf{v}^{\epsilon},\Delta_{q}]\cdot\nabla
\mathbf{v}^{\epsilon}\|_{L^2}+\|[m^{\epsilon},\Delta_{q}]\mathrm{div}\,
\mathbf{v}^{\epsilon}\|_{L^2}\nonumber\\&&+\|[m^{\epsilon},\Delta_{q}]\nabla
m^{\epsilon}\|_{L^2}+\frac{1}{\epsilon}\|\Delta_{q}(h(\epsilon
m^{\epsilon})\mathbf{v}^{\epsilon})\|_{L^2}\Big\}, \label{M-E15}
\end{eqnarray}
where the commutator $[f, g]: = fg - gf$ and $C$ is a uniform
positive constant independent of $\epsilon$.
\end{lem}
\begin{proof}
Applying the localization operator $\Delta_{q}$ to (\ref{M-E9})
gives
\begin{equation}
\left\{
\begin{array}{l}
\hspace{5mm}\partial_{t}\Delta_{q}m^{\epsilon} +
\bar{\psi}\epsilon^{-1}\mbox{div}\Delta_{q}\mathbf{v}^{\epsilon}
+(\mathbf{v}^{\epsilon}\cdot\nabla)\Delta_{q}m^{\epsilon}
\\=[\mathbf{v}^{\epsilon},\Delta_{q}]\cdot\nabla m^{\epsilon}
-\frac{\gamma-1}{2}m^{\epsilon}\Delta_{q}\mbox{div}\mathbf{v}^{\epsilon}+\frac{\gamma-1}{2}[m^{\epsilon},\Delta_{q}]\mbox{div}\mathbf{v}^{\epsilon},\\
\\\hspace{5mm}
\partial_{t}\Delta_{q}\mathbf{v}^{\epsilon}+
\bar{\psi}\epsilon^{-1}\nabla
\Delta_{q}m^{\epsilon}+(\mathbf{v}^{\epsilon}\cdot\nabla)\Delta_{q}\mathbf{v}^{\epsilon}\\=[\mathbf{v}^{\epsilon},\Delta_{q}]\cdot\nabla\mathbf{v}^{\epsilon}
-\frac{\gamma-1}{2}m^{\epsilon}\Delta_{q}\nabla
m^{\epsilon}+\frac{\gamma-1}{2}[m^{\epsilon},\Delta_{q}]\nabla
m^{\epsilon}\\\hspace{5mm}+\epsilon^{-1}\nabla\Delta_{q}\phi^{\epsilon}
-\Delta_{q}\mathbf{v}^{\epsilon}\\\\
\Delta_{q}\Delta\phi^{\epsilon}=\epsilon^{-1}\Delta_{q}h(\epsilon
m^{\epsilon}).
\end{array}
\right.\label{M-E16}
\end{equation}
Then, by multiplying  the first equation of (\ref{M-E16}) by
$\Delta_{q}m^{\epsilon}$, the second one by
$\Delta_{q}\mathbf{v}^{\epsilon}$ respectively, and adding the two
resulting equations together, then integrating the resulting
equations over $\mathbb{R}^N$, we get
\begin{eqnarray}
&&\frac{1}{2}\frac{d}{dt}\Big(\|\Delta_{q}m^{\epsilon}\|^2_{L^2} +
\|\Delta_{q}\mathbf{v}^{\epsilon}\|^2_{L^2}\Big)+\|\Delta_{q}\mathbf{v}^{\epsilon}\|^2_{L^2}
\nonumber\\
& = &
\int\mathrm{div}\mathbf{v}^{\epsilon}(|\Delta_{q}m^{\epsilon}|^2+|\Delta_{q}\mathbf{v}^{\epsilon}|^2)
\nonumber\\
    &&+\int([\mathbf{v}^{\epsilon},\Delta_{q}]\cdot\nabla
m^{\epsilon}\Delta_{q}m^{\epsilon}
+[\mathbf{v}^{\epsilon},\Delta_{q}]\cdot\nabla\mathbf{v}^{\epsilon}\Delta_{q}\mathbf{v}^{\epsilon})
\nonumber\\&&+\frac{\gamma-1}{2}\int\Delta_{q}m^{\epsilon}(\nabla
m^{\epsilon}\cdot\Delta_{q}\mathbf{v}^{\epsilon})
+\frac{\gamma-1}{2}
\int[m^{\epsilon},\Delta_{q}]\mbox{div}\mathbf{v}^{\epsilon}\cdot
\Delta_{q}m^{\epsilon}\nonumber\\
&&+\frac{\gamma-1}{2} \int[m^{\epsilon},\Delta_{q}]\nabla
m^{\epsilon}\cdot\Delta_{q}\mathbf{v}^{\epsilon}
+\frac{1}{\epsilon}\int\nabla\Delta_{q}\phi^{\epsilon}\cdot
\Delta_{q}\mathbf{v}^{\epsilon}.
 \label{M-E17}
\end{eqnarray}

In above equality (\ref{M-E17}), we may use the spectral
localization mass equation and Poisson equation in (\ref{M-E16}) in
order to eliminate the singularity from the electron-field term
similar to the idea in \cite{ACJP}, but this will cause very tedious
calculations. Here, we observe an equality
\begin{eqnarray}
\mathrm{div}\mathbf{v}^{\epsilon}=-\frac{\epsilon\mathrm{div}\nabla
\phi^{\epsilon}_{t}+\mathrm{div}(h(\epsilon
m^{\epsilon})\mathbf{v}^{\epsilon})}{\bar{n}}\label{M-E19}
\end{eqnarray}
following from the first equation and the third one in (\ref{M-E2})
under the symmetrization. Hence, we deduce that
\begin{eqnarray}
&&\frac{1}{\epsilon}\int\nabla\Delta_{q}
\phi^{\epsilon}\cdot\Delta_{q}\mathbf{v}^{\epsilon}\nonumber\\
&=&-\frac{1}{\epsilon}\int\Delta_{q}\phi^{\epsilon}
\Delta_{q}\mathrm{div}\mathbf{v}^{\epsilon}\nonumber\\
&=&\frac{1}{\bar{n}\epsilon}\int\Delta_{q}\phi^{\epsilon}\Delta_{q}
\Big(\epsilon\mathrm{div}\mathit\nabla\phi^{\epsilon}_{t}+\mathrm{div}(h(\epsilon
m^{\epsilon})\mathbf{v}^{\epsilon})\Big)\nonumber\\
&=&-\frac{1}{2\bar{n}}\frac{d}{dt}\|\Delta_{q}\nabla\phi^{\epsilon}\|_{L^2}^2-
\frac{1}{\bar{n}\epsilon}\int\Delta_{q}\nabla\phi^{\epsilon}
\Delta_{q}(h(\epsilon
m^{\epsilon})\mathbf{v}^{\epsilon}),\nonumber\\
&\leq&-\frac{1}{2\bar{n}}\frac{d}{dt}\|\Delta_{q}\nabla\phi^{\epsilon}\|_{L^2}^2+
\frac{1}{\bar{n}\epsilon}\|\Delta_{q}(h(\epsilon
m^{\epsilon})\mathbf{v}^{\epsilon})\|_{L^2}\|\Delta_{q}\nabla\phi^{\epsilon}\|_{L^2}.
\label{M-E20}
\end{eqnarray}

Together with (\ref{M-E17}) and (\ref{M-E20}), we arrive at
(\ref{M-E15}) immediately with the aid of Cauchy-Schwartz
inequality.
\end{proof}

In this position, we formulate an important skew-symmetry lemma
which has been well developed in \cite{CG,KY,Y2}, which is sometimes
referred to as the ``Kawashima condition".
\begin{lem}[Shizuta-Kawashima] \label{lem4.2}For all ~$\xi\in \mathbb{R}^{N},\
\xi\neq0$, there exists a real skew-symmetric smooth matrix $K(\xi)$
which is defined in the unit sphere $\textbf{S}^{N-1}$:
\begin{eqnarray}
K(\xi)=\left(%
\begin{array}{cc}
  0 & \frac{\xi^\top}{|\xi|} \\
  -\frac{\xi}{|\xi|} & 0 \\
\end{array}%
\right),\label{M-E21}
\end{eqnarray}
such that
\begin{eqnarray}
K(\xi)\sum_{j=1}^{N}\xi_{j}A^{\epsilon}_{j}(0)=\left(%
\begin{array}{cc}
  \frac{\bar{\psi}}{\epsilon}|\xi| & 0 \\
  0 & -\frac{\bar{\psi}}{\epsilon}\frac{\xi\otimes\xi}{|\xi|} \\
\end{array}%
\right),\label{M-E22}
\end{eqnarray}
where $A_{j}^{\epsilon}$ is the matrix appearing in the system
(\ref{M-E11}).
\end{lem}

Due to the skew-symmetry structure of the system (\ref{M-E11}), we
can develop some new frequency-localization estimates and avoid
performing the $t$-derivative to (\ref{M-E11}) as in \cite{FXZ}.
\begin{lem}\label{lem4.3}
If $(m^{\epsilon},\mathbf{v}^{\epsilon}, \nabla\phi^{\epsilon})\in
\mathcal{C}([0,T],B^{\sigma}_{2,1})\cap
\mathcal{C}^{1}([0,T],B^{\sigma-1}_{2,1})$ is a solution of
(\ref{M-E9})-(\ref{M-E10}) for any given $T>0$, then the following
estimates hold:
\begin{eqnarray}
&&\frac{\epsilon}{2}\frac{d}{dt}\mathrm{Im}\int|\xi|
\Big((\widehat{\Delta_{q}W^{\epsilon}})^{\ast}K(\xi)\widehat{\Delta_{q}W^{\epsilon}}\Big)d\xi
+\frac{\bar{\psi}}{2}2^{2q}\|\Delta_{q}m^{\epsilon}\|^2_{L^2}\nonumber\\
&\leq&
C2^{2q}\|\Delta_{q}\mathbf{v}^{\epsilon}\|^2_{L^2}+C\epsilon2^{q}\|\Delta_{q}W^{\epsilon}\|_{L^2}(\|\Delta_{q}\mathcal{G}\|_{L^2}
+\|m^{\epsilon}\|_{L^{\infty}}\|\Delta_{q}\mathrm{div}\mathbf{v}^{\epsilon}\|_{L^2}\nonumber\\&&+\|[m^{\epsilon},\Delta_{q}]\mathrm{div}\mathbf{v}^{\epsilon}\|_{L^2}
+\|m^{\epsilon}\|_{L^{\infty}}\|\Delta_{q}\nabla
m^{\epsilon}\|_{L^2}+\|[m^{\epsilon},\Delta_{q}]\nabla
m^{\epsilon}\|_{L^2}) \nonumber\\&&+C\|\Delta_{q}(H(\epsilon
m^{\epsilon})m^{\epsilon})\|_{L^2}\|\Delta_{q}m^{\epsilon}\|_{L^2}\
\ (q\geq0);\label{M-E23}\end{eqnarray}
\begin{eqnarray}
&&\frac{\epsilon}{2}\frac{d}{dt}\mathrm{Im}\int|\xi|
\Big((\widehat{\Delta_{-1}W^{\epsilon}})^{\ast}K(\xi)\widehat{\Delta_{-1}W^{\epsilon}}\Big)d\xi
+(A\gamma)^{-\frac{1}{2}}\bar{n}^{\frac{3-\gamma}{2}}\|\Delta_{-1}m^{\epsilon}\|^2_{L^2}\nonumber\\
&\leq&
C\|\Delta_{-1}\mathbf{v}^{\epsilon}\|^2_{L^2}+C\epsilon\|\Delta_{-1}W^{\epsilon}\|_{L^2}(\|\Delta_{-1}\mathcal{G}\|_{L^2}
+\|m^{\epsilon}\|_{L^{\infty}}\|\Delta_{-1}\mathrm{div}\mathbf{v}^{\epsilon}\|_{L^2}\nonumber\\&&+\|[m^{\epsilon},\Delta_{-1}]\mathrm{div}\mathbf{v}^{\epsilon}\|_{L^2}
+\|m^{\epsilon}\|_{L^{\infty}}\|\Delta_{-1}\nabla
m^{\epsilon}\|_{L^2}+\|[m^{\epsilon},\Delta_{-1}]\nabla
m^{\epsilon}\|_{L^2}) \nonumber\\&&+C\|\Delta_{-1}(H(\epsilon
m^{\epsilon})m^{\epsilon})\|_{L^2}\|\Delta_{-1}m^{\epsilon}\|_{L^2}.
\label{M-E24}\end{eqnarray} where the function $\mathcal{G}$ is
given by (\ref{M-E26}), $H(m)=\int_{0}^{1}h'(\varsigma
m)d\varsigma-(A\gamma)^{-\frac{1}{2}}\bar{n}^{\frac{3-\gamma}{2}}$
is a smooth function on $\{m|\varsigma m+\bar{h}>0,\
\varsigma\in[0,1]\}$ satisfying $H(0)=0$ and $C$ is a uniform
positive constant independent of $\epsilon$.
\end{lem}
\begin{proof}
The system (\ref{M-E11}) can be written as the linearized form
\begin{equation}\partial_{t}W^{\epsilon}+\sum_{j=1}^{N}A^{\epsilon}_{j}(0)\partial_{x_{j}}W^{\epsilon}
=\mathcal{G}+\left(%
\begin{array}{c}
  -\frac{\gamma-1}{2}m^{\epsilon}\mathrm{div}\, \mathbf{v}^{\epsilon}\\
   -\mathbf{v}^{\epsilon}-\frac{\gamma-1}{2}m^{\epsilon}\nabla m^{\epsilon}+\frac{1}{\epsilon}\nabla\phi^{\epsilon}\\
\end{array}%
\right),\label{M-E25}\end{equation}where
\begin{eqnarray}
\mathcal{G}=\sum_{j=1}^{N}\Big\{A^{\epsilon}_{j}(0)-A^{\epsilon}_{j}(\mathbf{v}^{\epsilon})\Big\}\partial_{x_{j}}W^{\epsilon}.\label{M-E26}
\end{eqnarray}
Applying the operator $\Delta_{q}$ to the system (\ref{M-E25}) gives
\begin{eqnarray}
    &&\partial_{t}\Delta_{q}W^{\epsilon}+\sum_{j=1}^{N}A^{\epsilon}_{j}(0)
    \partial_{x_{j}}\Delta_{q}W^{\epsilon}\nonumber\\
        &=&\Delta_{q}\mathcal{G}+
    \left(%
\begin{array}{c}
  -\frac{\gamma-1}{2}\Delta_{q}(m^{\epsilon}\mathrm{div}\, \mathbf{v}^{\epsilon}) \\
   -\Delta_{q}\mathbf{v}^{\epsilon}-\frac{\gamma-1}{2}\Delta_{q}(m^{\epsilon}\nabla m^{\epsilon})+\frac{1}{\epsilon}\Delta_{q}\nabla\phi^{\epsilon}\\
\end{array}%
\right),\label{M-E27}
\end{eqnarray}
By performing the Fourier transform with respect to the space
variable $x$ for (\ref{M-E27}) and multiplying the resulting
equation by
$-i\epsilon(\widehat{\Delta_{q}W^{\epsilon}})^{\ast}K(\xi)$($^{\ast}$
represents transpose and conjugate), then taking the real part of
each term in the equality, we can obtain
\begin{eqnarray}
&& \epsilon\mathrm{Im}
\Big((\widehat{\Delta_{q}W^{\epsilon}})^{\ast}K(\xi)\frac{d}{dt}\widehat{\Delta_{q}W^{\epsilon}}\Big)+\epsilon(\widehat{\Delta_{q}W^{\epsilon}})^{\ast}K(\xi)
\Big(\sum_{j=1}^{N}\xi_{j}A^{\epsilon}_{j}(0)\Big)\widehat{\Delta_{q}W^{\epsilon}}\nonumber\\
&=&
\epsilon\mathrm{Im}\Big((\widehat{\Delta_{q}W^{\epsilon}})^{\ast}K(\xi)(\widehat{\Delta_{q}\mathcal{G}})\Big)-\epsilon\mathrm{Im}
\Big((\overline{\widehat{\Delta_{q}m^{\epsilon}}})\frac{\xi^{\top}}{|\xi|}
\widehat{\Delta_{q}\mathbf{v}^{\epsilon}}\Big)
\nonumber\\&&+\mathrm{Im}
\Big((\overline{\widehat{\Delta_{q}m^{\epsilon}}})
\frac{\xi^{\top}}{|\xi|}\widehat{\Delta_{q}\nabla\phi^{\epsilon}}\Big)
+\frac{(\gamma-1)\epsilon}{2}\mathrm{Im}\Big(\overline{\widehat{\Delta_{q}\mathbf{v}^{\epsilon}}}\cdot\frac{\xi}{|\xi|}\widehat{(\Delta_{q}(m^{\epsilon}\mathrm{div}\,
\mathbf{v}^{\epsilon}))}\Big)\nonumber\\
    &&
-\frac{(\gamma-1)\epsilon}{2}\mathrm{Im}\Big(\overline{\widehat{\Delta_{q}m^{\epsilon}}}\frac{\xi^{\top}}{|\xi|}\widehat{(\Delta_{q}(m^{\epsilon}\nabla
m^{\epsilon}))}\Big).\label{M-E28}
\end{eqnarray}
Using the skew-symmetry of $K(\xi)$, we have
\begin{eqnarray}
\mathrm{Im}
\Big((\widehat{\Delta_{q}W^{\epsilon}})^{\ast}K(\xi)\frac{d}{dt}\widehat{\Delta_{q}W^{\epsilon}}\Big)=\frac{1}{2}\frac{d}{dt}\mathrm{Im}
\Big((\widehat{\Delta_{q}W^{\epsilon}})^{\ast}K(\xi)\widehat{\Delta_{q}W^{\epsilon}}\Big).\label{M-E29}
\end{eqnarray}
Substituting (\ref{M-E22}) into the second term on the left-hand
side of (\ref{M-E28}), it is not difficult to get
\begin{eqnarray}
&&\epsilon\mathrm{Im}
\Big((\widehat{\Delta_{q}W^{\epsilon}})^{\ast}K(\xi)\frac{d}{dt}\widehat{\Delta_{q}W^{\epsilon}}\Big)+\epsilon(\widehat{\Delta_{q}W^{\epsilon}})^{\ast}K(\xi)
\Big(\sum_{j=1}^{N}\xi_{j}A^{\epsilon}_{j}(0)\Big)\widehat{\Delta_{q}W^{\epsilon}}\nonumber\\&\geq&
\frac{\epsilon}{2}\frac{d}{dt}\mathrm{Im}
\Big((\widehat{\Delta_{q}W^{\epsilon}})^{\ast}K(\xi)\widehat{\Delta_{q}W^{\epsilon}}\Big)+\bar{\psi}|\xi||\widehat{\Delta_{q}W^{\epsilon}}|^2-2\bar{\psi}
|\xi||\widehat{\Delta_{q}\mathbf{v}^{\epsilon}}|^2.\label{M-E30}
\end{eqnarray}
With the help of Young inequality, the right-hand side of
(\ref{M-E28}) can be estimated as
\begin{eqnarray}
&&\epsilon\mathrm{Im}\Big((\widehat{\Delta_{q}W^{\epsilon}})^{\ast}K(\xi)(\widehat{\Delta_{q}\mathcal{G}})\Big)-\epsilon\mathrm{Im}
\Big((\overline{\widehat{\Delta_{q}m^{\epsilon}}})
\frac{\xi^{\top}}{|\xi|}\widehat{\Delta_{q}\mathbf{v}^{\epsilon}}\Big)
\nonumber\\
    &&+\mathrm{Im}
\Big((\overline{\widehat{\Delta_{q}m^{\epsilon}}})\frac{\xi^{\top}}{|\xi|}\widehat{\Delta_{q}\nabla\phi^{\epsilon}}\Big)
+\frac{(\gamma-1)\epsilon}{2}\mathrm{Im}\Big(\overline{\widehat{\Delta_{q}\mathbf{v}^{\epsilon}}}\cdot\frac{\xi}{|\xi|}\widehat{(\Delta_{q}(m^{\epsilon}\mathrm{div}\,
\mathbf{v}^{\epsilon}))}\Big)\nonumber\\
    &&
-\frac{(\gamma-1)\epsilon}{2}\mathrm{Im}\Big(\overline{\widehat{\Delta_{q}m^{\epsilon}}}\frac{\xi^{\top}}{|\xi|}\widehat{(\Delta_{q}(m^{\epsilon}\nabla
m^{\epsilon}))}\Big)\nonumber\\&\leq&
\frac{\bar{\psi}}{2}|\xi||\widehat{\Delta_{q}W^{\epsilon}}|^2+\frac{C}{|\xi|}|\widehat{\Delta_{q}\mathbf{v}^{\epsilon}}|^2
+\epsilon|\widehat{\Delta_{q}W^{\epsilon}}||\widehat{\Delta_{q}\mathcal{G}}|+C\epsilon|\widehat{\Delta_{q}\mathbf{v}^{\epsilon}}||\widehat{(\Delta_{q}(m^{\epsilon}\mathrm{div}\,
\mathbf{v}^{\epsilon}))}|
\nonumber\\&&+C\epsilon|\widehat{\Delta_{q}m^{\epsilon}}||\widehat{(\Delta_{q}(m^{\epsilon}\nabla
m^{\epsilon})})|+\mathrm{Im}
\Big((\overline{\widehat{\Delta_{q}m^{\epsilon}}})\frac{\xi^{\top}}{|\xi|}\widehat{\Delta_{q}\nabla\phi^{\epsilon}}\Big),
\label{M-E31}
\end{eqnarray}
where we have used the uniform boundedness of the matrix
$K(\xi)(\xi\neq 0)$. Combining the equality (\ref{M-E28}) and the
inequality (\ref{M-E30})-(\ref{M-E31}), we deduce
\begin{eqnarray}
&&\frac{\epsilon}{2}\frac{d}{dt}\mathrm{Im}
\Big((\widehat{\Delta_{q}W^{\epsilon}})^{\ast}K(\xi)\widehat{\Delta_{q}W^{\epsilon}}\Big)+\frac{\bar{\psi}}{2}|\xi||\widehat{\Delta_{q}W^{\epsilon}}|^2\nonumber\\&\leq&C\Big(|\xi|+\frac{1}{|\xi|}\Big)|\widehat{\Delta_{q}\mathbf{v}^{\epsilon}}|^2
+\epsilon|\widehat{\Delta_{q}W^{\epsilon}}||\widehat{\Delta_{q}\mathcal{G}}|+C\epsilon|\widehat{\Delta_{q}\mathbf{v}^{\epsilon}}||\widehat{(\Delta_{q}(m^{\epsilon}\mathrm{div}\,
\mathbf{v}^{\epsilon}))}|\nonumber\\&&+C\epsilon|\widehat{\Delta_{q}m^{\epsilon}}||\widehat{(\Delta_{q}(m^{\epsilon}\nabla
m^{\epsilon})})|+\mathrm{Im}
\Big((\overline{\widehat{\Delta_{q}m^{\epsilon}}})\frac{\xi^{\top}}{|\xi|}\widehat{\Delta_{q}\nabla\phi}\Big).\label{M-E32}
\end{eqnarray}
Multiplying (\ref{M-E32}) by $|\xi|$ and integrating it over
$\mathbb{R}^{N}$, using Plancherel's theorem, we obtain
\begin{eqnarray}
&&\frac{\epsilon}{2}\frac{d}{dt}\mathrm{Im}\int|\xi|
\Big((\widehat{\Delta_{q}W^{\epsilon}})^{\ast}K(\xi)\widehat{\Delta_{q}W^{\epsilon}}\Big)d\xi
+\frac{\bar{\psi}}{2}\|\Delta_{q}\nabla W^{\epsilon}\|^2_{L^2}\nonumber\\
&\leq&C(2^{2q}+1)\|\Delta_{q}\mathbf{v}^{\epsilon}\|^2_{L^2}+C\epsilon2^{q}\|\Delta_{q}W^{\epsilon}\|_{L^2}\|\Delta_{q}\mathcal{G}\|_{L^2}
\nonumber\\&&+C\epsilon2^{q}\|\Delta_{q}\mathbf{v}^{\epsilon}\|_{L^2}
\|\Delta_{q}(m^{\epsilon}\mathrm{div} \mathbf{v}^{\epsilon})\|_{L^2}
+C\epsilon2^{q}\|\Delta_{q}m^{\epsilon}\|_{L^2}\|\Delta_{q}(m^{\epsilon}\nabla
m^{\epsilon})\|_{L^2}\nonumber\\
    &&+\mathrm{Im}
\Big((\overline{\widehat{\Delta_{q}m^{\epsilon}}})\xi^{\top}\widehat{\Delta_{q}\nabla\phi^{\epsilon}}\Big)d\xi.
\label{M-E33}\end{eqnarray} Furthermore, the last term on the
right-hand side of (\ref{M-E33}) can be estimated as
\begin{eqnarray}&&
\mathrm{Im}
\Big((\overline{\widehat{\Delta_{q}m^{\epsilon}}})\xi^{\top}\widehat{\Delta_{q}\nabla\phi^{\epsilon}}\Big)d\xi\nonumber\\
&=&-\frac{\mathrm{\emph{i}}}{2}\int\Big((\overline{\widehat{\Delta_{q}m^{\epsilon}}})\xi^{\top}\widehat{\Delta_{q}\nabla\phi^{\epsilon}}\Big)d\xi
+\frac{\mathrm{\emph{i}}}{2}\int\Big((\widehat{\Delta_{q}m^{\epsilon}})\xi^{\top}\overline{\widehat{\Delta_{q}\nabla\phi^{\epsilon}}}\Big)d\xi\nonumber\\
&=&\frac{1}{2}\int(\overline{\widehat{\Delta_{q}\nabla
m^{\epsilon}}})\cdot\widehat{\Delta_{q}\nabla\phi^{\epsilon}}d\xi+\frac{1}{2}\int(\widehat{\Delta_{q}\nabla
m^{\epsilon}})\cdot\overline{\widehat{\Delta_{q}\nabla\phi^{\epsilon}}}d\xi\nonumber\\
 &=&\frac{(2\pi)^{N}}{2}\Big\{\int\overline{\Delta_{q}\nabla
 m^{\epsilon}}\cdot\Delta_{q}\nabla\phi^{\epsilon} dx+\int\Delta_{q}\nabla
 m^{\epsilon}\cdot\overline{\Delta_{q}\nabla\phi^{\epsilon}}dx\Big\}
\nonumber\\
 &=&-\frac{(2\pi)^{N}}{2}\Big\{\int\overline{\Delta_{q}m^{\epsilon}}\Delta_{q}\Delta\phi^{\epsilon} dx+\int\Delta_{q}m^{\epsilon}\overline{\Delta_{q}\Delta\phi^{\epsilon}}dx\Big\}\nonumber\\
&=&-\frac{(2\pi)^{N}}{2\epsilon}\Big\{\int\overline{\Delta_{q}m^{\epsilon}}\Delta_{q}(h(\epsilon m^{\epsilon})-h(0))dx+\int\Delta_{q}m^{\epsilon}\overline{\Delta_{q}(h(\epsilon m^{\epsilon})-h(0))}dx\Big\}\nonumber\\
&=&-(A\gamma)^{-\frac{1}{2}}\bar{n}^{\frac{3-\gamma}{2}}(2\pi)^{N}\|\Delta_{q}m^{\epsilon}\|^2_{L^2}\nonumber\\&&-\frac{(2\pi)^{N}}{2}\Big\{\int\overline{\Delta_{q}m^{\epsilon}}\Delta_{q}(H(\epsilon
m^{\epsilon})m^{\epsilon})dx
+\int\Delta_{q}m^{\epsilon}\overline{\Delta_{q}(H(\epsilon
m^{\epsilon})m^{\epsilon})}dx\Big\}, \label{M-E34}\end{eqnarray}
where $H(m)=\int_{0}^{1}h'(\varsigma
m)d\varsigma-(A\gamma)^{-\frac{1}{2}}\bar{n}^{\frac{3-\gamma}{2}}$
is a smooth function on $\{m|\frac{\gamma-1}{2}\varsigma
m+\bar{\psi}>0,\ \varsigma\in[0,1]\}$ satisfying $H(0)=0$.
Therefore, from (\ref{M-E33})-(\ref{M-E34}), we have
\begin{eqnarray}
&&\frac{\epsilon}{2}\frac{d}{dt}\mathrm{Im}\int|\xi|
\Big((\widehat{\Delta_{q}W^{\epsilon}})^{\ast}K(\xi)\widehat{\Delta_{q}W^{\epsilon}}\Big)d\xi
+\frac{\bar{\psi}}{2}\|\Delta_{q}\nabla
W^{\epsilon}\|^2_{L^2}\nonumber\\
    &&+
(A\gamma)^{-\frac{1}{2}}\bar{n}^{\frac{3-\gamma}{2}}(2\pi)^{N}\|\Delta_{q}m^{\epsilon}\|^2_{L^2}\nonumber\\
&\leq&
C(2^{2q}+1)\|\Delta_{q}\mathbf{v}^{\epsilon}\|^2_{L^2}+C\epsilon2^{q}\|\Delta_{q}W^{\epsilon}\|_{L^2}\|\Delta_{q}\mathcal{G}\|_{L^2}
\nonumber\\
    &&+C\epsilon2^{q}\|\Delta_{q}\mathbf{v}^{\epsilon}\|_{L^2}
\|\Delta_{q}(m^{\epsilon}\mathrm{div}\,
\mathbf{v}^{\epsilon})\|_{L^2}
+C\epsilon2^{q}\|\Delta_{q}m^{\epsilon}\|_{L^2}\|\Delta_{q}(m^{\epsilon}\nabla
m^{\epsilon})\|_{L^2}\nonumber\\
        &&+C\|\Delta_{q}(H(\epsilon
m^{\epsilon})m^{\epsilon})\|_{L^2}\|\Delta_{q}m^{\epsilon}\|_{L^2}.
\label{M-E35}
\end{eqnarray} In view of Lemma \ref{lem2.1}
$$\|\Delta_{q}\nabla f\|_{L^2}\approx
2^{q}\|\Delta_{q}f\|_{L^2}\ (q\geq0),$$ we get the estimate
(\ref{M-E23}) and (\ref{M-E24}) immediately.
\end{proof}
On the electron field $\nabla\phi$, we have the following \textit{a
priori} estimates.
\begin{lem}\label{lem4.4}
If $(m^{\epsilon},\mathbf{v}^{\epsilon}, \nabla\phi^{\epsilon})\in
\mathcal{C}([0,T],B^{\sigma}_{2,1})\cap
\mathcal{C}^{1}([0,T],B^{\sigma-1}_{2,1})$ is a solution of
(\ref{M-E9})-(\ref{M-E10}) for any given $T>0$, then
\begin{eqnarray}&&2^{2q}\|\Delta_{q}\nabla\phi^{\epsilon}\|^2_{L^2}
\label{M-E36}\\
&\leq&
C\Big((A\gamma)^{-\frac{1}{2}}\bar{n}^{\frac{3-\gamma}{2}}\|\Delta_{q}m^{\epsilon}\|_{L^2}
+\|\Delta_{q}(H(\epsilon
m^{\epsilon})m^{\epsilon})\|_{L^2}\Big)2^{q}\|\Delta_{q}\nabla\phi^{\epsilon}\|_{L^2}\
\ \ (q\geq0);\nonumber\end{eqnarray}
\begin{eqnarray}
&&-\epsilon\frac{d}{dt}\int\Delta_{-1}\nabla\phi^{\epsilon}\cdot\overline{\Delta_{-1}\mathbf{v}^{\epsilon}}
+(A\gamma)^{-\frac{1}{2}}\bar{n}^{\frac{3-\gamma}{2}}\bar{\psi}\|\Delta_{-1}m^{\epsilon}\|^2_{L^2}
+\|\Delta_{-1}\nabla\phi^{\epsilon}\|^2_{L^2}
\nonumber\\
&\leq&
C(\bar{n}\|\Delta_{-1}\mathbf{v}^{\epsilon}\|_{L^2}+\|\Delta_{-1}(h(\epsilon
m^{\epsilon})\mathbf{v}^{\epsilon})\|_{L^2})\|\Delta_{-1}\mathbf{v}^{\epsilon}\|_{L^2}\nonumber\\&&
+C\Big(\|\Delta_{-1}\mathbf{v}^{\epsilon}\|_{L^2}+\|\mathbf{v}^{\epsilon}\|_{L^{\infty}}\|\Delta_{-1}\nabla\mathbf{v}^{\epsilon}\|_{L^2}
+\|[\mathbf{v}^{\epsilon},\Delta_{-1}]\nabla\mathbf{v}^{\epsilon}\|_{L^2}\nonumber\\
&&+\|m^{\epsilon}\|_{L^{\infty}}\|\Delta_{-1}\nabla
m^{\epsilon}\|_{L^2} +\|[m^{\epsilon},\Delta_{-1}]\nabla
m^{\epsilon}\|_{L^2}\Big)\|\Delta_{-1}\nabla\phi^{\epsilon}\|_{L^2}\nonumber\\
    &&
+C\|\Delta_{-1}(H(\epsilon m^{\epsilon})m^{\epsilon})\|_{L^2}
\|\Delta_{-1}m^{\epsilon}\|_{L^2},\label{M-E37}
\end{eqnarray}
where $C$ is a uniform positive constant independent of $\epsilon$.
\end{lem}
\begin{proof}
By applying the localization operator $\Delta_{q}(q\geq0)$ to both
sides of $\mbox{div}\nabla\phi^{\epsilon}=\epsilon^{-1}h(\epsilon
m^{\epsilon})$ , integrating it over $ \mathbb{R}^{N}$ after
multiplying $\Delta_{q}\mbox{div}\nabla\phi^{\epsilon}$, and
noticing the irrotationality of $\nabla\phi^{\epsilon}$, we can
obtain (\ref{M-E36}) in virtue of H\"{o}lder's inequality.

 From (\ref{M-E2}) and (\ref{M-E5}), we get
\begin{eqnarray}
\nabla\phi^{\epsilon}_{t}=-\frac{1}{\epsilon}\nabla\Delta^{-1}\nabla\cdot\{h(\epsilon
m^{\epsilon})\mathbf{v}^{\epsilon}+\bar{n}\mathbf{v}^{\epsilon}\},\label{M-E38}
\end{eqnarray}
where the non-local term $\nabla\Delta^{-1}\nabla \cdot f$ is the
product of Riesz transforms on $f$. From (\ref{M-E9}) and
(\ref{M-E38}), we have
\begin{eqnarray}
&&-\epsilon\frac{d}{dt}\int\Delta_{-1}\nabla\phi^{\epsilon}\cdot\overline{\Delta_{-1}\mathbf{v}^{\epsilon}}\nonumber\\
&=&-\epsilon\int\Delta_{-1}\nabla\phi^{\epsilon}_{t}\cdot\overline{\Delta_{-1}\mathbf{v}^{\epsilon}}-\epsilon\int\Delta_{-1}\nabla\phi^{\epsilon}\cdot\overline{\Delta_{-1}\mathbf{v}^{\epsilon}_{t}}
\nonumber\\&=&\mathcal{I}+\int\nabla\Delta^{-1}\nabla\cdot\Delta_{-1}\{h(\epsilon
m^{\epsilon})\mathbf{v}^{\epsilon}+\bar{n}\mathbf{v}^{\epsilon}\}\overline{\Delta_{-1}\mathbf{v}^{\epsilon}}
\nonumber\\&&-\epsilon\int\Delta_{-1}\nabla\phi^{\epsilon}\cdot\Big(-\overline{\Delta_{-1}\mathbf{v}^{\epsilon}}-\mathbf{v}^{\epsilon}\overline{\Delta_{-1}\nabla\mathbf{v}^{\epsilon}}+\overline{[\mathbf{v}^{\epsilon},\Delta_{-1}]\nabla\mathbf{v}^{\epsilon}}
\nonumber\\&&-\frac{\gamma-1}{2}m\overline{\Delta_{-1}\nabla
m^{\epsilon}}+\frac{\gamma-1}{2}\overline{[m^{\epsilon},\Delta_{-1}]\nabla
m^{\epsilon}}+\frac{1}{\epsilon}\overline{\Delta_{-1}\nabla\phi^{\epsilon}}\Big)
\label{M-E39}
\end{eqnarray}
where $\mathcal{I}$ can be estimated as
\begin{eqnarray}
\mathcal{I}&=&\bar{\psi}\int\Delta_{-1}\nabla\phi^{\epsilon}\overline{\Delta_{-1}\nabla
m^{\epsilon}}\nonumber\\&=&-\bar{\psi}\int\Delta_{-1}\Delta\phi^{\epsilon}\overline{\Delta_{-1}
m^{\epsilon}}\nonumber\\&=&-\frac{\bar{\psi}}{\epsilon}\int\Delta_{-1}h(\epsilon
m^{\epsilon})\overline{\Delta_{-1}
m^{\epsilon}}\nonumber\\&=&-(A\gamma)^{-\frac{1}{2}}\bar{n}^{\frac{3-\gamma}{2}}\bar{\psi}\|\Delta_{-1}m^{\epsilon}\|^2_{L^2}
-\bar{\psi}\int\Delta_{-1}(H(\epsilon
m^{\epsilon})m^{\epsilon})\overline{\Delta_{-1}m^{\epsilon}}.
\label{M-E40}
\end{eqnarray}
Then using the $L^2$-boundedness of Riesz transform and H\"{o}lder's
inequality, we derive (\ref{M-E37}) immediately. \end{proof}

For the estimates of the commutators in (\ref{M-E15}) and
(\ref{M-E23})-(\ref{M-E24}) and (\ref{M-E37}), we have the following
conclusion.
\begin{lem}[see \cite{FXZ}]\label{lem4.5}
Let $s>0$ and $1<p<\infty$; then the following inequalities are
true:
\begin{eqnarray*}
&&2^{qs}\|[f,\Delta_{q}]\mathcal{A}g\|_{L^{p}}\nonumber\\&\leq&
\left\{
\begin{array}{l}
 Cc_{q}\|f\|_{B^{s}_{p,1}}\|g\|_{B^{s}_{p,1}},\ \ \
\ \ \ f,g\in B^{s}_{p,1},\ s=1+N/p,\\
 Cc_{q}\|f\|_{B^{s}_{p,1}}\|g\|_{B^{s+1}_{p,1}},\
 \ \ \ f\in B^{s}_{p,1},\ \ g\in B^{s+1}_{p,1}, \ s=N/p,\\
 Cc_{q}\|f\|_{B^{s+1}_{p,1}}\|g\|_{B^{s}_{p,1}},\ \
\ \ \ f\in B^{s+1}_{p,1},\ \ g\in B^{s}_{p,1}, \ s=N/p.
\end{array} \right.
\end{eqnarray*}
In particular, if $f=g$, then
\begin{eqnarray*}
2^{qs}\|[f,\Delta_{q}]\mathcal{A}g\|_{L^{p}} \leq Cc_{q}\|\nabla
f\|_{L^{\infty}}\|g\|_{B^{s}_{p,1}},\ s>0,
\end{eqnarray*}
where the operator $\mathcal{A}=\mathrm{div}$ or $\mathrm{\nabla}$,
$C$ is a harmless constant, and $c_{q}$ denotes a sequence such that
$\|(c_{q})\|_{ {l^{1}}}\leq 1.$
\end{lem}

With these lemmas for ready, now, we complete the proof of the
unform \textit{a priori} estimate (\ref{M-E14}).

\noindent\textit{\underline{Proof of Proposition \ref{prop4.1}.}}
Note that the \textit{a priori} assumption (\ref{M-E13}), we deduce
from the embedding inequality in Besov spaces that
\begin{eqnarray}
\sup_{0\leq t\leq T}
(\|(m^{\epsilon},\mathbf{v}^{\epsilon},\nabla\phi^{\epsilon})(\cdot,t)\|_{{W}^{1,\infty}})\leq
C\delta_{1}. \label{M-E41}\end{eqnarray} To ensure the smoothness of
functions $h(\epsilon m^{\epsilon})$ and $H(\epsilon m^{\epsilon})$,
together with the smallness of $\epsilon$, it suffices to choose
$0<\delta_{1}\leq\frac{\bar{\psi}}{(\gamma-1)C}$ such that
$$\frac{\gamma-1}{2}\epsilon m^{\epsilon}(t,x)+\bar{\psi}\geq \frac{\bar{\psi}}{2}>0,\ \ \ (t,x)\in[0,T]\times\mathbb{R}^{N}$$
and
$$\frac{\gamma-1}{2}\varsigma \epsilon m^{\epsilon}(t,x)+\bar{\psi}\geq\frac{\bar{\psi}}{2}>0, \ \varsigma\in[0,1],\ (t,x)\in[0,T]\times\mathbb{R}^{N}.$$

For the proof of Proposition \ref{prop4.1}, it can be divided into
the following high-frequency part and low-frequency part.
\begin{lem}[$q\geq0$]\label{lem4.6}
There exist some positive constants $K_{1}, K_{2}, \mu_{2}$
independent of $\epsilon$ such that the following estimate holds:
\begin{eqnarray}&&2^{q(\sigma-1)}\frac{d}{dt}\Big\{\frac{K_{1}}{2}2^{2q}\Big(\|\Delta_{q}m^{\epsilon}\|^2_{L^2}+\|\Delta_{q}\mathbf{v}^{\epsilon}\|^2_{L^2}+\frac{1}{\bar{n}}\|\Delta_{q}\nabla\phi^{\epsilon}\|^2_{L^2}\Big)
\nonumber\\&&+\frac{K_{2}\epsilon}{2}\mathrm{Im}\int|\xi|
\Big((\widehat{\Delta_{q}W^{\epsilon}})^{\ast}K(\xi)\widehat{\Delta_{q}W^{\epsilon}}\Big)d\xi\Big\}^{1/2}
\nonumber\\
    &&+\mu_{2}2^{q\sigma}\Big(\|\Delta_{q}m^{\epsilon}\|_{L^2}
+\|\Delta_{q}\mathbf{v}^{\epsilon}\|_{L^2}
+\|\Delta_{q}\nabla\phi^{\epsilon}\|_{L^2}\Big)
\nonumber\\
    &\leq& C2^{q\sigma}\|(W^{\epsilon},\nabla
W^{\epsilon})\|_{L^\infty}\|\Delta_{q}W^{\epsilon}\|_{L^2}
+Cc_{q}\|W^{\epsilon}\|^2_{B^{\sigma}_{2,1}}
+\frac{C2^{q\sigma}}{\epsilon}\|\Delta_{q}(h(\epsilon
m^{\epsilon})\mathbf{v}^{\epsilon})\|_{L^2}
\nonumber\\
    &&+C2^{q(\sigma-1)}\|\Delta_{q}\mathcal{G}\|_{L^2}+C2^{q\sigma}\|\Delta_{q}(H(\epsilon
m^{\epsilon})m^{\epsilon})\|_{L^2} , \label{M-E42}
\end{eqnarray}
where $K_{1},\ K_{2}$ are given by (\ref{M-E44}) and $C$ is a
uniform positive constant independent of $\epsilon$.
\end{lem}
\begin{proof}
Combining (\ref{M-E15}), (\ref{M-E23}) and (\ref{M-E36}), we have
\begin{eqnarray}&&\frac{d}{dt}\Big\{\frac{K_{1}}{2}2^{2q}\Big(\|\Delta_{q}m^{\epsilon}\|^2_{L^2}+\|\Delta_{q}\mathbf{v}^{\epsilon}\|^2_{L^2}+\frac{1}{\bar{n}}\|\Delta_{q}\nabla\phi^{\epsilon}\|^2_{L^2}\Big)
\nonumber\\&&+\frac{K_{2}\epsilon}{2}\mathrm{Im}\int|\xi|
\Big((\widehat{\Delta_{q}W^{\epsilon}})^{\ast}K(\xi)\widehat{\Delta_{q}W^{\epsilon}}\Big)d\xi\Big\}\nonumber\\&&+K_{1}2^{2q}\|\Delta_{q}\mathbf{v}^{\epsilon}\|^2_{L^2}
+\frac{K_{2}\bar{\psi}}{2}2^{2q}\|\Delta_{q}m^{\epsilon}\|^2_{L^2}+K_{3}2^{2q}\|\Delta_{q}\nabla\phi^{\epsilon}\|^2_{L^2}\nonumber\\
&\leq&
CK_{1}2^{2q}\|(\Delta_{q}m^{\epsilon},\Delta_{q}\mathbf{v}^{\epsilon},\Delta_{q}\nabla\phi^{\epsilon})\|_{L^2}
\Big\{\|(\nabla
m^{\epsilon},\nabla\mathbf{v}^{\epsilon})\|_{L^{\infty}}\|(\Delta_{q}m^{\epsilon},\Delta_{q}\mathbf{v}^{\epsilon})\|_{L^2}
\nonumber\\&&+\|[\mathbf{v}^{\epsilon},\Delta_{q}]\cdot\nabla
m^{\epsilon}\|_{L^2}+\|[\mathbf{v}^{\epsilon},\Delta_{q}]\cdot\nabla
\mathbf{v}^{\epsilon}\|_{L^2}+\|[m^{\epsilon},\Delta_{q}]\mathrm{div}\,
\mathbf{v}^{\epsilon}\|_{L^2}\nonumber\\&&+\|[m^{\epsilon},\Delta_{q}]\nabla
m^{\epsilon}\|_{L^2}+\frac{1}{\epsilon}\|\Delta_{q}(h(\epsilon
m^{\epsilon})\mathbf{v}^{\epsilon})\|_{L^2}\Big\}
+CK_{2}2^{2q}\|\Delta_{q}\mathbf{v}^{\epsilon}\|^2_{L^2}\nonumber\\&&+CK_{2}\epsilon2^{q}\|\Delta_{q}W^{\epsilon}\|_{L^2}(\|\Delta_{q}\mathcal{G}\|_{L^2}
+\|m^{\epsilon}\|_{L^{\infty}}\|\Delta_{q}\mathrm{div}\mathbf{v}^{\epsilon}\|_{L^2}\nonumber\\&&+\|[m^{\epsilon},\Delta_{q}]\mathrm{div}\mathbf{v}^{\epsilon}\|_{L^2}
+\|m^{\epsilon}\|_{L^{\infty}}\|\Delta_{q}\nabla
m^{\epsilon}\|_{L^2}+\|[m^{\epsilon},\Delta_{q}]\nabla
m^{\epsilon}\|_{L^2}) \nonumber\\&&+CK_{2}\|\Delta_{q}(H(\epsilon
m^{\epsilon})m^{\epsilon})\|_{L^2}\|\Delta_{q}m^{\epsilon}\|_{L^2}\label{M-E43}\\
&&+CK_{3}\Big((A\gamma)^{-\frac{1}{2}}\bar{n}^{\frac{3-\gamma}{2}}\|\Delta_{q}m^{\epsilon}\|_{L^2}
+\|\Delta_{q}(H(\epsilon
m^{\epsilon})m^{\epsilon})\|_{L^2}\Big)2^{q}\|\Delta_{q}\nabla\phi^{\epsilon}\|_{L^2},
\nonumber
\end{eqnarray}
where the uniform positive constants $K_{1},K_{2}$ and $K_{3}$
(independent of $\epsilon$) satisfy
\begin{eqnarray}
K_{2}=\frac{K_{1}}{4C},\ \
K_{3}=\frac{A\gamma\bar{\psi}}{4C^2\bar{n}^{3-\gamma}}K_{2}.
\label{M-E44}
\end{eqnarray}
Due to
\begin{eqnarray}
&&\Big|\frac{K_{2}\epsilon}{2}\mathrm{Im}\int|\xi|
\Big((\widehat{\Delta_{q}W^{\epsilon}})^{\ast}K(\xi)\widehat{\Delta_{q}W^{\epsilon}}\Big)d\xi\Big|
\nonumber\\&\leq&\frac{CK_{2}}{2}2^{2q}(\|\Delta_{q}m^{\epsilon}\|^2_{L^2}
+\|\Delta_{q}\mathbf{v}^{\epsilon}\|^2_{L^2})\ (0<\epsilon\leq1),
 \label{M-E45}
\end{eqnarray}
so we introduce these uniform constants in order to ensure
\begin{eqnarray}
&&\frac{K_{1}}{2}2^{2q}\Big(\|\Delta_{q}m^{\epsilon}\|^2_{L^2}+\|\Delta_{q}\mathbf{v}^{\epsilon}\|^2_{L^2}+\frac{1}{\bar{n}}\|\Delta_{q}\nabla\phi^{\epsilon}\|^2_{L^2}\Big)
\nonumber\\&&+\frac{K_{2}\epsilon}{2}\mathrm{Im}\int|\xi|
\Big((\widehat{\Delta_{q}W^{\epsilon}})^{\ast}K(\xi)\widehat{\Delta_{q}W^{\epsilon}}\Big)d\xi
\nonumber\\&\approx&
2^{2q}\Big(\|\Delta_{q}m^{\epsilon}\|^2_{L^2}+\|\Delta_{q}\mathbf{v}^{\epsilon}\|^2_{L^2}+\|\Delta_{q}\nabla\phi^{\epsilon}\|^2_{L^2}\Big)\label{M-E46}
\end{eqnarray}
and eliminate the quadratic terms
    $$
    CK_{3}(A\gamma)^{-\frac{1}{2}}\bar{n}^{\frac{3-\gamma}{2}}2^{q}\|\Delta_{q}m^{\epsilon}\|_{L^2}\|\Delta_{q}\nabla\phi^{\epsilon}\|_{L^2}
\ \textrm{ and }\
    CK_{2}2^{2q}\|\Delta_{q}\mathbf{v}^{\epsilon}\|^2_{L^2}
    $$
 in the
right-hand side of (\ref{M-E43}) with the aid of Young's inequality,
for similar details, see \cite{FXZ}. Dividing the resulting
inequality by
\begin{eqnarray*}&&\Big\{\frac{K_{1}}{2}2^{2q}\Big(\|\Delta_{q}m^{\epsilon}\|^2_{L^2}+\|\Delta_{q}\mathbf{v}^{\epsilon}\|^2_{L^2}+\frac{1}{\bar{n}}\|\Delta_{q}\nabla\phi^{\epsilon}\|^2_{L^2}\Big)
\nonumber\\&&+\frac{K_{2}\epsilon}{2}\mathrm{Im}\int|\xi|
\Big((\widehat{\Delta_{q}W^{\epsilon}})^{\ast}K(\xi)\widehat{\Delta_{q}W^{\epsilon}}\Big)d\xi\Big\}^{1/2}\end{eqnarray*}
after eliminating the quadratic terms, then multiplying the factor
$2^{q(\sigma-1)}$ on the both sides of inequality, we arrive at
(\ref{M-E42}) immediately with the help of Lemma \ref{lem4.5} and
the smallness of $\epsilon(0<\epsilon\leq1)$. \end{proof}

Similarly, we also have the \textit{a priori} estimate for the case
of low frequency.
\begin{lem}$(q=-1)$\label{lem4.7}
There exist some positive constants $\bar{K}_{1}, \bar{K}_{2},
\bar{K}_{3}$ and $\mu_{3}$ independent of $\epsilon$, such that the
following estimate holds:
\begin{eqnarray}&&2^{-(\sigma-1)}\frac{d}{dt}\Big\{\frac{\bar{K}_{1}}{2}2^{-2}\Big(\|\Delta_{-1}m^{\epsilon}\|^2_{L^2}
+\|\Delta_{-1}\mathbf{v}^{\epsilon}\|^2_{L^2}+\frac{1}{\bar{n}}\|\Delta_{-1}\nabla\phi^{\epsilon}\|^2_{L^2}\Big)
\nonumber\\&&+\frac{\bar{K}_{2}\epsilon}{2}\mathrm{Im}\int|\xi|
\Big((\widehat{\Delta_{-1}W^{\epsilon}})^{\ast}K(\xi)\widehat{\Delta_{-1}W^{\epsilon}}\Big)d\xi-\bar{K}_{3}\epsilon\int\Delta_{-1}\nabla\phi^{\epsilon}\cdot
\overline{\Delta_{-1}\mathbf{v}^{\epsilon}}dx\Big\}^{1/2}
\nonumber\\&&+\mu_{3}2^{-\sigma}\Big(\|\Delta_{-1}m^{\epsilon}\|_{L^2}+\|\Delta_{-1}\mathbf{v}^{\epsilon}\|_{L^2}+\|\Delta_{-1}\nabla\phi^{\epsilon}\|_{L^2}\Big)
\nonumber\\&\leq& C2^{-\sigma}\|(W^{\epsilon},\nabla
W^{\epsilon})\|_{L^\infty}\|\Delta_{-1}W^{\epsilon}\|_{L^2}
+Cc_{-1}\|W^{\epsilon}\|^2_{B^{\sigma}_{2,1}}
\nonumber\\
    &&+\frac{C}{\epsilon}2^{-\sigma}\|\Delta_{-1}(h(\epsilon
m^{\epsilon})\mathbf{v}^{\epsilon})\|_{L^2}
+C2^{-(\sigma-1)}\|\Delta_{-1}\mathcal{G}\|_{L^2}\nonumber\\
    &&+C2^{-\sigma}\|\Delta_{-1}(H(\epsilon
m^{\epsilon})m^{\epsilon})\|_{L^2},\label{M-E47}
\end{eqnarray}
where $C$ is a uniform positive constant independent of $\epsilon$.
\end{lem}
\begin{rmk}\label{rem4.1}
Similar to the proof of Lemma \ref{lem4.6}, the constants
$\bar{K}_{1}, \bar{K}_{2}, \bar{K}_{3}$ are introduced to ensure
that
\begin{eqnarray}
&&\frac{\bar{K}_{1}}{2}2^{-2}\Big(\|\Delta_{-1}m^{\epsilon}\|^2_{L^2}
+\|\Delta_{-1}\mathbf{v}^{\epsilon}\|^2_{L^2}+\frac{1}{\bar{n}}\|\Delta_{-1}\nabla\phi^{\epsilon}\|^2_{L^2}\Big)
\nonumber\\&&+\frac{\bar{K}_{2}\epsilon}{2}\mathrm{Im}\int|\xi|
\Big((\widehat{\Delta_{-1}W^{\epsilon}})^{\ast}K(\xi)\widehat{\Delta_{-1}W^{\epsilon}}\Big)d\xi-\bar{K}_{3}\epsilon\int\Delta_{-1}\nabla\phi^{\epsilon}\cdot
\overline{\Delta_{-1}\mathbf{v}^{\epsilon}}dx \nonumber\\&\approx&
2^{-2}\Big(\|\Delta_{-1}m^{\epsilon}\|^2_{L^2}
+\|\Delta_{-1}\mathbf{v}^{\epsilon}\|^2_{L^2}+\|\Delta_{-1}\nabla\phi^{\epsilon}\|^2_{L^2}\Big).
\label{M-E48}
\end{eqnarray}
and eliminate some quadratic terms appearing in the right-hand side
of inequality (\ref{M-E47}).
\end{rmk}
Summing (\ref{M-E42}) on $q \in \mathbb{N}\cup\{0\}$ and adding
(\ref{M-E47}) together, according to the \textit{a priori}
 assumption (\ref{M-E13}), (\ref{M-E41}) and Moser's estimates
 (Proposition \ref{prop2.2}), we obtain the following differential inequality:
\begin{eqnarray}\frac{d}{dt} {Q}(t)+\mu_{4}\|(m^{\epsilon},\mathbf{v}^{\epsilon},\nabla\phi^{\epsilon})\|_{B^{\sigma}_{2,1}}\leq
C\delta_{1}\|(m^{\epsilon},\mathbf{v}^{\epsilon},\nabla\phi^{\epsilon})\|_{B^{\sigma}_{2,1}},\label{M-E49}\end{eqnarray}
where
\begin{eqnarray}
 {Q}(t)&=&\sum_{q\geq0}2^{q(\sigma-1)}\Big\{\frac{K_{1}}{2}2^{2q}\Big(\|\Delta_{q}m^{\epsilon}\|^2_{L^2}+\|\Delta_{q}\mathbf{v}^{\epsilon}\|^2_{L^2}+\frac{1}{\bar{n}}\|\Delta_{q}\nabla\phi^{\epsilon}\|^2_{L^2}\Big)
\nonumber\\&&+\frac{K_{2}\epsilon}{2}\mathrm{Im}\int|\xi|
\Big((\widehat{\Delta_{q}W^{\epsilon}})^{\ast}K(\xi)\widehat{\Delta_{q}W^{\epsilon}}\Big)d\xi\Big\}^{1/2}
\nonumber\\&&+\Big\{\frac{\bar{K}_{1}}{2}2^{-2}\Big(\|\Delta_{-1}m^{\epsilon}\|^2_{L^2}
+\|\Delta_{-1}\mathbf{v}^{\epsilon}\|^2_{L^2}+\frac{1}{\bar{n}}\|\Delta_{-1}\nabla\phi^{\epsilon}\|^2_{L^2}\Big)
\nonumber\\&&+\frac{\bar{K}_{2}\epsilon}{2}\mathrm{Im}\int|\xi|
\Big((\widehat{\Delta_{-1}W^{\epsilon}})^{\ast}K(\xi)\widehat{\Delta_{-1}W^{\epsilon}}\Big)d\xi
\nonumber\\
&&-\bar{K}_{3}\int\Delta_{-1}\nabla\phi^{\epsilon}\cdot
\overline{\Delta_{-1}\mathbf{v}^{\epsilon}}dx\Big\}^{1/2}
\label{M-E50}
\end{eqnarray}
and $\mu_{4}$ is some positive constant. Note that
\begin{eqnarray}
 {Q}(t)\approx\|(m^{\epsilon},\mathbf{v}^{\epsilon},
\nabla\phi^{\epsilon})(\cdot,t)\|_{B^{\sigma}_{2,1}},\ \ t\geq0,
\label{M-E51}
\end{eqnarray}
and by choosing $\delta_{1}=\min\{\frac{\mu_{4}}{2C},
\frac{\bar{\psi}}{(\gamma-1)C}\}$, we get
\begin{eqnarray}&&\|(m^{\epsilon},\mathbf{v}^{\epsilon},
\nabla\phi^{\epsilon})(\cdot,t)\|_{B^{\sigma}_{2,1}}\leq
C\|\Big(\frac{m_{0}}{\epsilon},\mathbf{v}_{0},\frac{\textbf{e}_{0}}{\epsilon}\Big)\|_{B^{\sigma}_{2,1}}\exp(-\mu_{1}t),\label{M-E52}
\end{eqnarray}
where we have used the Gronwall's inequality,
$\mu_{1}:=\frac{\mu_{4}}{2}$. This is just the inequality
(\ref{M-E14}).

Hence, the proof of Proposition \ref{prop4.1} is complete. \qed

\section{Zero-electron-mass limit}\setcounter{equation}{0}
In this section, our first aim at deducing a frequency-localization
Strichartz-type estimate which is the crucial ingredient of the
zero-electron-mass limit. For this end, we need to give a detailed
analysis of the equation of acoustics (\ref{M-E55}) with the aid of
the semigroup formulation, and obtain some dispersive estimates.
Then, by using the classical $TT^{*}$ argument, we obtain the
desired Strichartz-type estimate.

\subsection{The linearized system}
From (\ref{M-E9}), we have
\begin{equation}
    \left\{\begin{array}{l}
      m^{\epsilon}_t+\bar{\psi}\frac{div \mathbf{v}^{\epsilon}}{\epsilon}=F,\\
    \mathbf{v}^{\epsilon}_t+\mathbf{v}^{\epsilon}+\bar{\psi}\frac{\nabla
    m^{\epsilon}}{\epsilon}-\frac{h'(0)\nabla\Delta^{-1}m^{\epsilon}}{\epsilon}=G,
    \end{array}
    \right.\label{M-E53}
        \end{equation}
 where
        $F=-\mathbf{v}^{\epsilon}\cdot\nabla m^{\epsilon}-\frac{\gamma-1}{2}m^{\epsilon}\mathrm{div}\mathbf{v}^{\epsilon}$,
        $G=-\mathbf{v}^{\epsilon}\cdot \nabla \mathbf{v}^{\epsilon}-\frac{\gamma-1}{2}m^{\epsilon}\nabla m^{\epsilon}
        +\frac{\nabla \Delta^{-1}(h(\epsilon m^{\epsilon})-h'(0)\epsilon m^{\epsilon})}{\epsilon^2}$.
Following from the idea in \cite{Danchin}, we split the velocity
into a divergence-free part $\mathcal{P}\mathbf{v}^{\epsilon}$ and a
gradient part $\mathcal{Q}\mathbf{v}$, where
$\mathcal{P}=I-\nabla\Delta^{-1}\mathrm{div}$ and
$\mathcal{Q}=I-\mathcal{P}=\nabla\Delta^{-1}\mathrm{div}$. Then, the
system (\ref{M-E53}) reads
\begin{equation}\left\{\begin{array}{l}
      m^{\epsilon}_t+\bar{\psi}\frac{div Q\mathbf{v}^{\epsilon}}{\epsilon}=F,\\
    \mathcal{Q}\mathbf{v}^{\epsilon}_t+\mathcal{Q}\mathbf{v}^{\epsilon}+\bar{\psi}\frac{\nabla
    m^{\epsilon}}{\epsilon}-\frac{h'(0)\nabla\Delta^{-1}m^{\epsilon}}{\epsilon}
    =\mathcal{Q}G,\\
    \mathcal{P}\mathbf{v}^{\epsilon}_t+\mathcal{P}\mathbf{v}^{\epsilon}=\mathcal{P}G,\\
        \mathbf{v}^{\epsilon}=\mathcal{P}\mathbf{v}^{\epsilon}+
        \mathcal{Q}\mathbf{v}^{\epsilon}.
      \end{array}
    \right.\label{M-E54}
        \end{equation}
Because of
$\|\mathcal{Q}\mathbf{v}^{\epsilon}\|_{B^\sigma_{2,1}}\approx\|
d^{\epsilon}\|_{B^\sigma_{2,1}}$, where
$d^{\epsilon}=\Lambda^{-1}\mathrm{div}\mathcal{Q}\mathbf{v}^{\epsilon}$
with $\mathcal{F}(\Lambda^{-1}f)=|\xi|^{-1}\mathcal{F}f$, we shall
investigate carefully the following mixed linear equation of
acoustics:
    \begin{equation}\left\{\begin{array}{l}
      m^{\epsilon}_t+\bar{\psi}\frac{\Lambda d^{\epsilon}}{\epsilon}=F,\\
    d^{\epsilon}_t+d^{\epsilon}-\bar{\psi}\frac{\Lambda
    m^{\epsilon}}{\epsilon}-\frac{h'(0)\Lambda^{-1}m^{\epsilon}}{\epsilon}
    =\Lambda^{-1}\mathrm{div}\mathcal{Q}G,\\
        (m^{\epsilon},d^{\epsilon})|_{t=0}=(m^{\epsilon}_0,d^{\epsilon}_0),
     \end{array}
    \right.\label{M-E55}
        \end{equation}
where $d^{\epsilon}_0:=\Lambda^{-1}\mathrm{div}\mathcal{Q}
\mathbf{v}^{\epsilon}_0$.

According to the semigroup theory for evolutional equation, the
solutions $(m^{\epsilon},d^{\epsilon})$ to the   linear initial
value problem (\ref{M-E55}) can be expressed for
$U^{\epsilon}=(m^{\epsilon},d^{\epsilon})^\top$ as
    \begin{equation}
      U^{\epsilon}_t=BU^{\epsilon}+(F,\Lambda^{-1}\mathrm{div}\mathcal{Q}G)^\top,\ U^{\epsilon}(0)=U^{\epsilon}_0=(m^{\epsilon}_0,d^{\epsilon}_0)^\top,\ t\geq0,\label{M-E56}
    \end{equation}
which gives rise to
    \begin{equation}
      U^{\epsilon}(t)=S(t)U^{\epsilon}_0+\int^t_0S(t-\tau)(F,\Lambda^{-1}\mathrm{div}\mathcal{Q}G)^\top d\tau
      ,\ t\geq0, \label{M-E57}
    \end{equation}
where $S(t)U^{\epsilon}_0:=e^{tB}U^{\epsilon}_0$. Then, we analyze
the differential operator $B$ by means of its Fourier expression
$A(\xi)$ and show the long time properties of the semigroup $S(t)$.
Applying the Fourier transform to the system (\ref{M-E56}) with
$F=0$ and $G=0$, we get
            \begin{equation}
              \partial_t\widehat{U^{\epsilon}}
              =A(\xi)\widehat{U^{\epsilon}},\ \widehat{U^{\epsilon}}(0)=\widehat{U^{\epsilon}_0}, \label{M-E58}
            \end{equation}
where
$\widehat{U^{\epsilon}}(t)=\widehat{U^{\epsilon}}(\xi,t)=\mathcal{F}U^{\epsilon}(\xi,t)$,
$\xi=(\xi_1,\ldots,\xi_N)^\top$ and $A(\xi)$ is defined as
    \begin{equation}
      A(\xi)=\left(\begin{array}
      {cc}
        0 &-\frac{\bar{\psi}|\xi|}{\epsilon}\\
             \frac{\bar{\psi}|\xi|}{\epsilon}+\frac{h'(0)}{|\xi|\epsilon}&-1
    \end{array}
    \right). \label{M-E59}
    \end{equation}
The eigenvalues of the matrix $A(\xi)$ are computed from the
determinant
    $$
        \mathrm{det}(A(\xi)- {\lambda} I)=
        = \left|\begin{array}
      {cc}
        -\lambda &-\frac{\bar{\psi}|\xi|}{\epsilon}\\
             \frac{\bar{\psi}|\xi|}{\epsilon}+\frac{h'(0)}{|\xi|\epsilon}&-1-\lambda
    \end{array}
    \right|=0,
    $$
which implies
    $$
       \lambda_\pm=-\frac{1}{2}\pm \frac{i}{\epsilon}\sqrt{\bar{\psi}^2|\xi|^2+\bar{\psi}h'(0)-\frac{1}{4}\epsilon^2}
    :=-\frac{1}{2}\pm \frac{i}{\epsilon}\lambda.
        $$
   Hence, the semigroup $e^{tA}$ exhibits the expression
    $$
    e^{tA}=e^{\lambda_+t}\frac{A-\lambda_-I}{\lambda_+-\lambda_-}
    +e^{\lambda_-t}\frac{A-\lambda_+I}{\lambda_--\lambda_+}
    :=e^{\lambda_+t}P_++e^{\lambda_-t}P_-.
    $$
After a direct computation, we can verify the exact expression about
the Fourier transform $\widehat{e^{tB}} $ of the Green's function $
e^{tB}$ as
                \begin{eqnarray}
                &  &\widehat{e^{tB}}:=e^{tA}=e^{\lambda_+t}P_++e^{\lambda_-t}P_-\nonumber\\
                &=&\left(\begin{array}
      {cc}
        \frac{-\lambda_-e^{\lambda_+t}+\lambda_+e^{\lambda_-t}}{\lambda_+-\lambda_-}
        &-\frac{\bar{\psi}|\xi|}{\epsilon}\cdot\frac{e^{\lambda_+t}-e^{\lambda_-t}}{\lambda_+-\lambda_-}\\
             \left(\frac{\bar{\psi}|\xi|}{\epsilon}+\frac{h'(0)}{|\xi|\epsilon}\right)
             \frac{e^{\lambda_+t}-e^{\lambda_-t}}{\lambda_+-\lambda_-}
             &-\frac{e^{\lambda_+t}-e^{\lambda_-t}}{\lambda_+-\lambda_-}
             + \frac{-\lambda_-e^{\lambda_+t}+\lambda_+e^{\lambda_-t}}{\lambda_+-\lambda_-}
    \end{array}
    \right).\label{M-E60}
                \end{eqnarray}

\begin{lem}[Dispersive estimate]\label{L3.1}
With the above notations, when $\epsilon$ is small enough, we get
the following estimate:
\begin{equation} |I_{1,k}(t,\tau,z )|\leq
\epsilon C{2}^{ (N
-1)k}\max\{1,2^{-k}\}e^{-\frac{1}{2}t}\min\Big\{\frac{2^k}{2^k+1},\tau^{-\frac{1}{2}}\Big\}
,\label{M-E61}\end{equation}
 where  $$
    I_{1,k }(t,\tau,z)=\int_{\mathbb{R}^N}e^{i\xi\cdot z}\psi (2^{-k}|\xi|)
    \frac{e^{-\frac{1}{2}t}e^{i\tau\lambda}}{\lambda_+-\lambda_-}
    d\xi,
    $$
    and smooth function $\psi$ satisfies
$\mathrm{supp}\psi(x)\in\{x\in\mathbb{R} |\frac{1}{6} \leq x\leq 3
\}$ and $\psi(x)|_{\frac{5}{6}\leq|x|\leq \frac{12}{5}}=1$, $C>0$
denotes a uniform constant independent of $k$ and $\epsilon$.
\end{lem}
\begin{proof}
 Using the rotation invariant in $\xi$, we
restrict ourselves to the case when $z_{2}=\ldots=z_{N}=0$. The
estimate will follow from  the stationary phase theorem. Denoting
$\alpha(\xi):=-
\partial_{\xi_{2}}(\lambda)=-\frac{\bar{\psi}^2\xi_{2}}{\lambda}$, we
introduce the following differential operator
$$\mathcal{L}:=\frac{1+i\alpha(\xi)\partial_{\xi_{2}}}{1+\tau \alpha^2(\xi)},$$
which acts on the $\xi_{2}$ variable, and satisfies
$\mathcal{L}(e^{i\tau \lambda})=e^{i\tau \lambda}$. Easy computation
yields
    $$ ^{\top}\mathcal{L}=\frac{1}{1+\tau \alpha^2(\xi)}-i(\partial_{\xi_{2}}\alpha)
    \frac{1-\tau \alpha^2}{(1+\tau \alpha^2)^2}-\frac{i\alpha}{1+\tau
    \alpha^2}\partial_{\xi_{2}}.$$
Using the integration by parts, we obtain
$$I_{1,k}(t,\tau,z )=
\int_{\mathbb{R}^N} {}^{\top}\mathcal{L}\left[
\psi(2^{-k}|\xi|)\frac{e^{-\frac{1}{2}t}}{\lambda_+-\lambda_-}
\right]e^{i\tau\lambda +iz_{1}\xi_{1}}d \xi,$$ where
    \begin{eqnarray*}
    && ^{\top}\mathcal{L}\left[\psi(2^{-k}|\xi|)\frac{e^{-\frac{1}{2}t}}{\lambda_+-\lambda_-}
    \right]\\
& = &\left(\frac{1}{1+\tau
\alpha^2(\xi)}-i(\partial_{\xi_{2}}\alpha) \frac{1-\tau
\alpha^2}{(1+\tau
\alpha^2)^2}\right)\psi(2^{-k}|\xi|)\frac{e^{-\frac{1}{2}t}}{\lambda_+-\lambda_-}\\&
&-\frac{i\alpha}{1+\tau
\alpha^2}\partial_{\xi_{2}}\left(\psi(2^{-k}|\xi|)\frac{e^{-\frac{1}{2}t}}{\lambda_+-\lambda_-}
\right).
\end{eqnarray*}
Because of $\frac{1}{6}2^{k} \leq |\xi| \leq 3 \cdot2^k$, when
$\epsilon$ is small enough satisfying
    $$
    \bar{\psi}^2|\xi|^2 -\frac{1}{8}\epsilon^2>
    \frac{1}{2}\bar{\psi}^2|\xi|^2,  \ \bar{\psi}h'(0)-\frac{1}{8}\epsilon^2>
    \frac{1}{2}\bar{\psi}h'(0)
    \textrm{ and } 2^k\epsilon\leq 1,$$
    we have
        $$
        |\alpha|\leq C,
        $$
$$\frac{1}{1+\tau
\alpha^2}+ \frac{|\alpha|}{1+\tau \alpha^2} \leq \frac{C
}{1+(2^{k}+1)^{-2}\tau \xi_{2}^2 },$$
$$ \frac{|\partial_{\xi_2}\alpha||1-\tau \alpha^2|}{(1+\tau
\alpha^2)^2} \leq \frac{C (2^{k}+1)^{-1}}{1+(2^{k}+1)^{-2}\tau
\xi_{2}^2 },$$
$$
\left|\psi(2^{-k}|\xi|)\frac{e^{-\frac{1}{2}t}}{\lambda_+-\lambda_-}
\right|\leq C\epsilon (2^{k}+1)^{-1} e^{-\frac{1}{2}t},
$$
$$
|\partial_{\xi_{2}}(\psi(2^{-k}|\xi|)\frac{e^{-\frac{1}{2}t}}{\lambda_+-\lambda_-}
)|\leq C\epsilon 2^{- k}(2^{k}+1)^{-1} e^{-\frac{1}{2}t}.
$$
An easy computation shows that
\begin{eqnarray*}
&&|I_{1,k}(t,\tau,z )|\\
    &\leq& \epsilon
C(2^{k}+1)^{-1}\max\{1,2^{-k}\}e^{-\frac{1}{2}t}
\int_{\frac{1}{6}2^{k}\leq|\xi|\leq
3\cdot 2^k}\frac{d \xi}{1+(2^{k}+1)^{-2}\tau \xi_{2}^2}\\
&\leq & \epsilon C {2}^{ (N
-1)k}\max\{1,2^{-k}\}e^{-\frac{1}{2}t}\min\Big\{\frac{2^k}{2^k+1},\tau^{-\frac{1}{2}}\Big\}.
\end{eqnarray*}
Hence the estimate (\ref{M-E61}) holds. \end{proof}

Similarly, we can obtain the following lemma and  omit the proofs
for brevity.
\begin{lem}[Dispersive estimates]\label{L3.2}
With the above notations, when $\epsilon$ is small enough, we get
the following estimates:
\begin{equation} |I_{2,k}(t,\tau,z )|\leq
\epsilon C{2}^{ (N
-1)k}\max\{1,2^{-k}\}e^{-\frac{1}{2}t}\min\Big\{\frac{2^k}{2^k+1},\tau^{-\frac{1}{2}}\Big\}
, \label{M-E62}\end{equation}
\begin{equation}  |I_{3,k}(t,\tau,z )|\leq
C{2}^{ Nk}\max\{
1,2^{-2k}\}e^{-\frac{1}{2}t}\min\Big\{\frac{2^k}{2^k+1},\tau^{-\frac{1}{2}}\Big\},\label{M-E63}\end{equation}
\begin{equation} |I_{4,k}(t,\tau,z )|\leq
  C{2}^{  N
 k}\max\{1,2^{-k}\}e^{-\frac{1}{2}t}\min\Big\{\frac{2^k}{2^k+1},\tau^{-\frac{1}{2}}\Big\},\label{M-E64}
\end{equation}
\begin{equation} |I_{5,k}(t,\tau,z )|\leq
  C{2}^{ ( N-2)
 k}\max\{1,2^{-k}\}e^{-\frac{1}{2}t}\min\Big\{\frac{2^k}{2^k+1},\tau^{-\frac{1}{2}}\Big\}, \label{M-E65}
\end{equation}
 where  $$
    I_{2,k}(t,\tau,z)=\int_{\mathbb{R}^N}e^{i\xi\cdot z}\psi(2^{-k}|\xi|)
    \frac{e^{-\frac{1}{2}t}e^{-i\tau\lambda}}{\lambda_+-\lambda_-}
    d\xi,
    $$
     $$
    I_{3,k}(t,\tau,z)=\int_{\mathbb{R}^N}e^{i\xi\cdot z}\psi(2^{-k}|\xi|)
    \frac{\lambda_\mp e^{-\frac{1}{2}t}e^{\pm i\tau\lambda}}{\lambda_+-\lambda_-}
    d\xi,
    $$
        $$
    I_{4,k}(t,\tau,z)=\int_{\mathbb{R}^N}e^{i\xi\cdot z}\psi(2^{-k}|\xi|)\frac{\bar{\psi}|\xi|}{\epsilon}
    \frac{ e^{-\frac{1}{2}t}e^{\pm i\tau\lambda}}{\lambda_+-\lambda_-}
    d\xi,
    $$
  $$
    I_{5,k}(t,\tau,z)=\int_{\mathbb{R}^N}e^{i\xi\cdot z}\psi(2^{-k}|\xi|)\frac{h'(0)}{|\xi|\epsilon}
    \frac{ e^{-\frac{1}{2}t}e^{\pm i\tau\lambda}}{\lambda_+-\lambda_-}
    d\xi,
    $$
   and $C>0$ denotes a uniform constant
independent of $k$ and $\epsilon$.
\end{lem}

\begin{prop}[Strichartz-type estimate]\label{p3.1}
Suppose $U^{\epsilon}$ is the solution of the  system (\ref{M-E56})
 with
$$\mathrm{supp}\hat{U}_{0}\cup \left\{\bigcup_{t\geq 0}\mathrm{supp}
(\hat{F}(t),\hat{G}(t))^\top\right\}\in
 \mathcal{C}_{k}=\left\{
 \xi\in\mathbb{R}^N\Big|\frac{5}{6}2^{k}\leq|\xi|\leq \frac{12}{5}\cdot 2^k
 \right\}.$$
When $\epsilon$ is small enough, we get the following estimate
\begin{equation}\|U^{\epsilon}\|_{L^1(\mathbb{R}^{+};L_{x}^{p})}
\underline{}\leq\left\{\begin{array}{ll}
C{2}^\frac{Nk(p-2)}{2p}\epsilon^{\frac{p-2
}{4p}}(\|U^{\epsilon}_{0}\|_{L^2}
+\|(F,G)\|_{L^1(\mathbb{R}^{+};L^2)})&
k\geq0,\\
C{2}^\frac{(N-4)k(p-2)}{2p}\epsilon^{\frac{p-2}{4p}}(\|U^{\epsilon}_{0}\|_{L^2}
+\|(F,G)\|_{L^1(\mathbb{R}^{+};L^2)}) &k<0,
\end{array}\right.  \label{M-E66}\end{equation}
where $p\in [2,+\infty]$.
\end{prop}
\begin{proof}
  Duhamel's formula enable us to restrict our attentions to the case
$(F,G)=0$. if we can obtain
\begin{equation}\|S(t)U^{\epsilon}_{0}\|_{L^1(\mathbb{R}^{+};L_{x}^{p})}
\leq \left\{\begin{array}{ll}
C{2}^\frac{Nk(p-2)}{2p}\epsilon^{\frac{p-2
}{4p}}\|U^{\epsilon}_{0}\|_{L^2}&
k\geq0,\\
C{2}^\frac{(N-4)k(p-2)}{2p}\epsilon^{\frac{p-2}{4p}}\|U^{\epsilon}_{0}\|_{L^2}
&k<0,
\end{array}\right.
\label{M-E67}
\end{equation}
by the same computations as that in \cite{Chemin2006}, we can get
the estimate on $$\int_{0}^{t}S({t-\tau})
(F,\Lambda^{-1}\mathrm{div}\mathcal{Q}G)^\top(\tau)d\tau$$ by Fubini
theorem and interpolation. Finally, we can get (\ref{M-E66}).

In the following, we mainly prove the estimate (\ref{M-E67}). For
simplicity, we only estimate the term
    $$
    A_{1,k}d^{\epsilon}_0=\int_{\mathbb{R}^N}I_{1,k}(t,\tau,x-y)d^{\epsilon}_0(y)dy,
    $$
 where  $$
    I_{1,k}(t,\tau,z)=\int_{\mathbb{R}^N}e^{i\xi\cdot z}\psi(2^{-k}|\xi|)
    \frac{e^{-\frac{1}{2}t}e^{i\tau\lambda}}{\lambda_+-\lambda_-}
    d\xi,
    $$
  $\tau=\frac{t}{\epsilon}$, and $\psi$ satisfies
supp$\psi(x)\in\{x\in\mathbb{R} |\frac{1}{6} \leq x\leq 3 \}$ and
$\psi(x)|_{\frac{5}{6}\leq|x|\leq \frac{12}{5}}=1$.

    Now we shall use the
$TT^{*}$ argument. Define
$$B:=\{a \in \mathcal{D}(\mathbb{R}^{+}\times \mathbb{R}^{N}),\|a\|
_{L^{\infty}(\mathbb{R}^{+};L^1 )}\leq 1\},$$ we have
\begin{eqnarray*}
\|
{A_{1,k}}d^{\epsilon}_{0}\|_{L^1(\mathbb{R}^{+};{L_{x}^{\infty}})}
 &=& \sup_{a \in B}\int _{\mathbb{R}^{+}\times
 \mathbb{R}^{2N}}I_{1,k}(t,\tau,x-y)d^{\epsilon}_{0}(y)
 a(t,x)dt dxdy\\
 & \leq & \|d^{\epsilon}_{0}\|_{L^2}\sup_{a \in B}\|\int _{\mathbb{R}^{+}}\breve{I}_{1,k}
 (t,\tau,\cdot)\ast
 a(t,\cdot)dt\|_{L^2},
 \end{eqnarray*}
where $\breve{b}(x)=b(-x)$. Let
$$\Phi := \|\int _{\mathbb{R}^{+}}\breve{I}_{1,k}
 (t,\tau, \cdot)\ast a(t,\cdot)dt\|_{L^2}.$$
 By the Fourier-Planchered theorem, we have
\begin{eqnarray*}
\Phi^2&=&C\int_{(\mathbb{R}^{+})^2\times
\mathbb{R}^N}\hat{I}_{1,k}(t,\tau,-\xi)\hat{a}(t,\xi)
\bar{\hat{I}}_{1,k}(s,\frac{s}{\epsilon},-\xi)\bar{\hat{a}}(s,\xi)dt
ds d\xi.
\end{eqnarray*}
Note that the following identity holds
\begin{eqnarray*}\hat{I}_{1,k}(t,\tau,-\xi )\bar{\hat{I}}_{1,k}(s,\frac{s}{\epsilon},-\xi)
=\hat{I}_{1,k}(t+s,\frac{t-s}{\epsilon},-\xi)\frac{
\psi(2^{-k}|\xi|)}{\lambda_+-\lambda_-},\end{eqnarray*} from Lemma
\ref{L3.1}, it follows that
\begin{eqnarray*}
&&\Phi^2\\
 &=&C\int_{(\mathbb{R}^{+})^2\times \mathbb{R}^N}\frac{
\psi(2^{-k}|\xi|)}{\lambda_+-\lambda_-}\mathcal{F}(\breve{I}_{1,k}
 (t+s,\frac{t-s}{\epsilon},\cdot )\ast
 a(t,\cdot )) \bar{\hat{a}}(s,\xi )dt ds d\xi \\
 &=& C\int_{(\mathbb{R}^{+})^2\times
\mathbb{R}^N}(\breve{I}_{1,k}
 (t+s,\frac{t-s}{\epsilon},\cdot )\ast
 a(t,\cdot ))(x ) \mathcal{F}^{-1}\Big\{\frac{
\psi(2^{-k}|\xi|)}{\lambda_+-\lambda_-}
\hat{a}\Big\}(s,x )dt ds dx \\
& \leq & C\epsilon (2^{k}+1)^{-1}
\int_{(\mathbb{R}^{+})^2}\|\breve{I}_{1,k}
 (t+s,\frac{t-s}{\epsilon},\cdot )\ast
a(t,\cdot ) \|_{L_{x }^{\infty} }\|a(s,\cdot )\|_{L_{x }^1 }dt ds\\
&\leq& C\epsilon (2^{k}+1)^{-1}\int_{(\mathbb{R}^{+})^2}
\|\breve{I}_{1,k}
 (t+s,\frac{t-s}{\epsilon},\cdot )\|_{L_x^{\infty}}
 \|a(t,\cdot )\|_{L_{x}^1}\|a(s,\cdot )\|_{L_{x}^1}dt ds
\\
 &\leq& \epsilon^2 C{2}^{(N-2)k}
 \int_{(\mathbb{R}^{+})^2}\min\Big\{\frac{2^k}{2^k+1},\frac{\epsilon^{\frac{1}{2}}}{(t-s)^{\frac{1}{2}}}\Big\}
 e^{-\frac{1}{2}(t+s)}dt
 ds\\
 &\leq& C {2}^{(N-2)k} \epsilon^{\frac{5}{2}}.
\end{eqnarray*}
Then we have
    $$
    \|A_{1,k}d^{\epsilon}_0\|_{L^1(\mathbb{R}^{+};L^{\infty})}
\leq C \epsilon^{\frac{5}{4}}2^\frac{(N-2)k}{2}
\|d^{\epsilon}_{0}\|_{L^2}.
    $$
Similarly, from Lemma \ref{L3.2}, we can estimate the other terms in
$S(t)U^{\epsilon}_0$ and obtain
$$\|U^{\epsilon}\|_{L^1(\mathbb{R}^{+};L^{\infty})}
\leq\left\{\begin{array}{ll}
C{2}^\frac{Nk}{2}\epsilon^{\frac{1}{4}}\|U^{\epsilon}_{0}\|_{L^2}&
k\geq0,\\
C{2}^\frac{(N-4)k}{2}\epsilon^{\frac{1}{4}}\|U^{\epsilon}_{0}\|_{L^2}
&k<0.
\end{array}\right.
$$
Similarly, we can obtain that the
$L^{1}({\mathbb{R}^{+};L^2}(\mathbb{R}^N)) $ norm of $U^{\epsilon}$
is bounded, uniformly in $\epsilon$,  by interpolation, we get
(\ref{M-E67}). \end{proof}

\subsection{Global convergence}
\begin{prop}\label{P-5.2} Suppose $(m^{\epsilon},\mathbf{v}^\epsilon,\nabla\phi^\epsilon)$
is the solution of the  system
    (\ref{M-E9})-(\ref{M-E10}), then we have
     \begin{equation}
    \|(m^{\epsilon},\mathcal{Q}\mathbf{v}^{\epsilon})\|_{L^1
(\mathbb{R}^{+};B^\frac{N}{p}_{p,1})}
 \leq \left\{\begin{array}{ll}
 C\epsilon^{\frac{(p-2)^2N}{8p(3p-4)}},
 & 1-\frac{(p-2)(N-4)}{2p}>0,\\
  C_\beta \epsilon^{\frac{p-2-\beta}{4p}}, &  1-\frac{(p-2)(N-4)}{2p}=0, \\
  C \epsilon^{\frac{p-2}{4p}}, &  1-\frac{(p-2)(N-4)}{2p}<0,
 \end{array}\label{M-E68}
 \right.
    \end{equation}
      \begin{equation}
    \|\nabla\phi^{\epsilon}\|_{L^1
(\mathbb{R}^{+};B^\frac{N}{p}_{p,1})}
 \leq \left\{\begin{array}{ll}
 C\epsilon^{\frac{(p-2)^2N}{32p( p-1 )}},
 & 2-\frac{(p-2)(N-4)}{2p}>0,\\
  C_\beta \epsilon^{\frac{p-2-\beta}{4p}}, &  2-\frac{(p-2)(N-4)}{2p}=0, \\
  C \epsilon^{\frac{p-2}{4p}}, &  2-\frac{(p-2)(N-4)}{2p}<0,
 \end{array}\label{M-E69}
 \right.
    \end{equation}
where $p\in[2,\infty]$ and $\beta\in(0,p-2)$.
\end{prop}
\begin{proof}
  From Proposition \ref{prop4.1}, we get
    $
    {m}^\epsilon, {\mathbf{v}}^\epsilon,\nabla  {\phi}^\epsilon
    $ uniformly bounded in $
       L^\infty_t{B^{\sigma}_{2,1}} \cap L^1_t{B^{\sigma}_{2,1}} $.
Let $P_{\geq M}f=\mathcal{F}^{-1}(\chi_{|\xi|\geq M}\hat{f})$, then
we have
    \begin{eqnarray}
\|P_{\geq M}(m^\epsilon,\mathcal{Q}{\mathbf{v}}^\epsilon,\nabla
{\phi}^\epsilon)\|_{L^1_tB^\frac{N}{p}_{p,1}}&\leq& C
     \|P_{\geq M}(m^\epsilon,\mathcal{Q}{\mathbf{v}}^\epsilon,\nabla
{\phi}^\epsilon)\|_{L^1_tB^\frac{N}{2}_{2,1}}\nonumber\\
&   \leq &
     CM^{-1} \|P_{\geq M}(m^\epsilon,{\mathbf{v}}^\epsilon,\nabla
{\phi}^\epsilon)\|_{L^1_tB^\sigma_{2,1}},
     \ \label{M-E70}
    \end{eqnarray}
where $p\in[2,\infty]$. Let $P_{\leq
\frac{1}{M}}f=\mathcal{F}^{-1}(\chi_{|\xi|\leq
\frac{1}{M}}\hat{f})$, then we obtain
    \begin{eqnarray}
    \|P_{\leq \frac{1}{J}}(m^\epsilon,\mathcal{Q}{\mathbf{v}}^\epsilon,\nabla
{\phi}^\epsilon)\|_{L^1_tL^p}&\leq& C
    J^{-\frac{(p-2)N}{2p}} \|P_{\leq \frac{1}{J}}(m^\epsilon,\mathcal{Q}{\mathbf{v}}^\epsilon,\nabla
{\phi}^\epsilon)\|_{L^1_tL^2}\nonumber\\
&\leq&
     CJ^{-\frac{(p-2)N}{2p}} \|
     (m^\epsilon,{\mathbf{v}}^\epsilon,\nabla
{\phi}^\epsilon)\|_{L^1_tL^2}.\label{M-E71}
    \end{eqnarray}
where $p\in[2,\infty]$.

From (\ref{M-E66}), we have
\begin{eqnarray}
&&\|\dot{\Delta}_k(m^{\epsilon},d^{\epsilon})
\|_{L^1(\mathbb{R}^{+};L^{p})} \label{M-E72}\\
    &\leq&
\left\{\begin{array}{ll} C{2}^\frac{Nk(p-2)}{2p}\epsilon^{\frac{p-2
}{4p}}\|\dot{\Delta}_k(m_{0},d_0)\|_{L^2}
+\|\dot{\Delta}_k(F,G)\|_{L^1(\mathbb{R}^{+};L^2)})&
k\geq0,\\
C{2}^\frac{(N-4)k(p-2)}{2p}\epsilon^{\frac{p-2}{4p}}\|\dot{\Delta}_k(m_{0},d_0)\|_{L^2}
+\|\dot{\Delta}_k(F,G)\|_{L^1(\mathbb{R}^{+};L^2)}) &k<0,
\end{array}\right.\nonumber
\end{eqnarray}
 where $p\in [2,+\infty]$. From Theorem \ref{thm1.1} and the classical estimate
 in Besov space for the product of two
functions $B^{\frac{N}{2} }_{2,1}\times B^\frac{
N}{2}_{2,1}\hookrightarrow B^{\frac{N}{2} }_{2,1} $, we have
    $$
    \|\dot{\Delta}_k(F,G)\|_{L^1(\mathbb{R}^{+};L^2)}\leq Cc_k
    2^{-\frac{Nk}{2}},\ k\geq0,
    $$
and
    $$
    \|\dot{\Delta}_k(F,G)\|_{L^1(\mathbb{R}^{+};L^2)}\leq C2^{-k},\
    k<0,
    $$
 where $C$ is independent of $\epsilon$ and $\sum c_k\leq 1$. Let $P_{A\leq\cdot\leq
B}f=\mathcal{F}^{-1}(\chi_{A\leq|\xi|\leq B}\hat{f})$. Then, we have
    \begin{equation}\|P_{\frac{1}{J}\leq\cdot\leq
1}(m^{\epsilon},d^{\epsilon})\|_{L^1(\mathbb{R}^{+};L^{p})}
 \leq \left\{\begin{array}{ll}
 C{J}^{1-\frac{(p-2)(N-4)}{2p}}\epsilon^{\frac{p-2}{4p}},
 & 1-\frac{(p-2)(N-4)}{2p}>0,\\
  C\ln{J}\epsilon^{\frac{p-2}{4p}},
 & 1-\frac{(p-2)(N-4)}{2p}=0,\\
  C \epsilon^{\frac{p-2}{4p}}, &  1-\frac{(p-2)(N-4)}{2p}<0,
 \end{array}
 \right.\label{M-E73}
\end{equation}
and
     \begin{equation}\|P_{1\leq\cdot\leq
M}(m^{\epsilon},\mathcal{Q}\mathbf{v}^{\epsilon})\|_{L^1
(\mathbb{R}^{+};B^\frac{N}{p}_{p,1})}
 \leq C\epsilon^{\frac{p-2}{4p}},\label{M-E74}
\end{equation}
where $p\in[2,\infty]$. From (\ref{M-E70})--(\ref{M-E71}) and
(\ref{M-E73})--(\ref{M-E74}), we have
    \begin{equation}
    \|(m^{\epsilon},\mathcal{Q}\mathbf{v}^{\epsilon})\|_{L^1
(\mathbb{R}^{+};B^\frac{N}{p}_{p,1})}
 \leq \left\{\begin{array}{ll}
 C\epsilon^{\frac{(p-2)^2N}{8p(3p-4)}},
 & 1-\frac{(p-2)(N-4)}{2p}>0,\\
 C_\beta\epsilon^{\frac{p-2-\beta}{4p}},
 & 1-\frac{(p-2)(N-4)}{2p}=0,  \\
  C \epsilon^{\frac{p-2}{4p}}, &  1-\frac{(p-2)(N-4)}{2p}<0,
 \end{array}
 \right.\label{M-E75}
    \end{equation}
where $p\in[2,\infty]$, $\beta\in(0,p-2)$,
$J=\epsilon^{-\frac{p-2}{4(3p-4)}}$ and
$M=\epsilon^{-\frac{(p-2)^2N}{8p(3p-4)}}$ when
$1-\frac{(p-2)(N-4)}{2p}>0$,
$J=\epsilon^{-\frac{p-2-\beta}{2N(p-2)}}$ and
$M=\epsilon^{-\frac{p-2-\beta}{4p}}$ when
$1-\frac{(p-2)(N-4)}{2p}=0$, $J=\epsilon^{-\frac{1}{2N}}$ and
$M=\epsilon^{-\frac{p-2}{4p}}$ when $1-\frac{(p-2)(N-4)}{2p}<0$.

Similarly, we have
     \begin{equation}\|P_{\frac{1}{J}\leq\cdot\leq
1}\nabla \phi^{\epsilon} \|_{L^1(\mathbb{R}^{+};L^{p})}
 \leq \left\{\begin{array}{ll}
 C{J}^{2-\frac{(p-2)(N-4)}{2p}}\epsilon^{\frac{p-2}{4p}},
 & 2-\frac{(p-2)(N-4)}{2p}>0,\\
 C\ln{J}\epsilon^{\frac{p-2}{4p}},
 & 2-\frac{(p-2)(N-4)}{2p}=0,\\
  C \epsilon^{\frac{p-2}{4p}}, &  2-\frac{(p-2)(N-4)}{2p}<0,
 \end{array}
 \right. \label{M-E76}
\end{equation}
     \begin{equation}\|P_{1\leq\cdot\leq
M}\nabla\phi^\epsilon\|_{L^1 (\mathbb{R}^{+};B^\frac{N}{p}_{p,1})}
 \leq C\epsilon^{\frac{p-2}{4p}}, \label{M-E77}
\end{equation}
and
    \begin{equation}
    \|\nabla\phi^{\epsilon}\|_{L^1
(\mathbb{R}^{+};B^\frac{N}{p}_{p,1})}
 \leq \left\{\begin{array}{ll}
 C\epsilon^{\frac{(p-2)^2N}{32p( p-1 )}},
 & 2-\frac{(p-2)(N-4)}{2p}>0,\\
  C_\beta \epsilon^{\frac{p-2-\beta}{4p}}, &  2-\frac{(p-2)(N-4)}{2p}=0, \\
  C \epsilon^{\frac{p-2}{4p}}, &  2-\frac{(p-2)(N-4)}{2p}<0,
 \end{array}
 \right. \label{M-E78}
    \end{equation}
where $p\in[2,\infty]$ and $\beta\in(0,p-2)$. \end{proof}

Now, we consider the global well-posedness of incompressible Euler
equations (\ref{M-E4}). Following from the standard
frequency-localization method, we can easily obtain the following
result.
\begin{thm}
  There exists $\delta>0$, such that
  if $\|\mathcal{P}\mathbf{v}_0\|_{B^\sigma_{2,1}}\leq\delta$, then there
  exists a unique solution $\mathbf{u}$ to the incompressible Euler
  equations
   (\ref{M-E4}) satisfying
        \begin{equation}
          \|\mathbf{u}\|_{L^\infty(\mathbb{R}^+;B^\sigma_{2,1})\cap L^1(\mathbb{R}^+;B^\sigma_{2,1})}
          \leq C \|\mathcal{P}v_0\|_{B^\sigma_{2,1}}. \label{M-E79}
        \end{equation}
\end{thm}

\begin{prop}\label{P-5.3}Suppose $(m^{\epsilon},\mathbf{v}^\epsilon, \nabla\phi^\epsilon)$
is the solution of the system
    (\ref{M-E9})-(\ref{M-E10}), then
        \begin{eqnarray}
     && \|\mathcal{P}\mathbf{v}^\epsilon-\mathbf{u}\|_{L^\infty([0,T];B^\frac{N}{p}_{p,1})\cap
     L^1([0,T];B^\frac{N}{p}_{p,1})} \nonumber\\
      &\leq&\left\{\begin{array}{ll}
 C\epsilon^{\frac{(p-2)^2N}{4p(3p-4)(N+2)}},
 &  1-\frac{(p-2)(N-4)}{2p}>0   \textrm{ and }2\leq p\leq N,\\
 C\epsilon^{\frac{(p-2)^2N^2}{4p^2(3p-4)(N+2)}},
 &  1-\frac{(p-2)(N-4)}{2p}>0   \textrm{ and }N<p\leq \infty,\\
   C_\beta \epsilon^{\frac{p-2-\beta}{2p(N+2)}}, &  1-\frac{(p-2)(N-4)}{2p}=0,\ \beta\in(0,p-2), \\
  C \epsilon^{\frac{p-2}{2p(N+2)}}, &
  1-\frac{(p-2)(N-4)}{2p}<0 \textrm{ and }2\leq p\leq N,\\
  C \epsilon^{\frac{N(p-2)}{2p^2(N+2)}}, &  1-\frac{(p-2)(N-4)}{2p}<0 \textrm{ and
  }N<p\leq\infty.
 \end{array}\label{M-E80}
 \right.
 \end{eqnarray}
\end{prop}
\begin{proof}
Let $w=\mathcal{P}\mathbf{v}^\epsilon-\mathbf{u}$. From
(\ref{M-E54})$_3$ and (\ref{M-E4}), we have
    \begin{equation}
    \left\{\begin{array}{l}
       \partial_tw
        +w=H,\\
         w|_{t=0}=0,
         \end{array}\right. \label{M-E81}
    \end{equation}
where $H=-\mathcal{P}(w\cdot\nabla \mathbf{u})
-\mathcal{P}(\mathbf{v}^\epsilon\cdot\nabla w)
-\mathcal{P}(\mathcal{Q}\mathbf{v}^\epsilon\cdot\nabla \mathbf{u})
-\mathcal{P}(\mathbf{v}^\epsilon\cdot\nabla
\mathcal{Q}\mathbf{v}^\epsilon)$. Then, by multiplying  by $w$ and
integrating the resulting equations over $\mathbb{R}^N$, we have
    \begin{eqnarray*}
      &&\|w\|_{L^\infty([0,T];L^2)}
      +      \|w\|_{L^1([0,T];L^2)}\\
      &\leq& C\int^T_0(\|w\|_{L^2}(\|\nabla \mathbf{u}\|_{L^\infty}+\|\nabla
      \mathbf{v}^\epsilon\|_{L^\infty})
      +\|\mathcal{Q}\mathbf{v}^\epsilon\|_{L^\infty}
      \|\nabla \mathbf{u}\|_{L^2}\\
      &&+\|\mathbf{v}^\epsilon\|_{L^\frac{2p}{p-2}}
      \|\nabla
\mathcal{Q}\mathbf{v}^\epsilon\|_{L^p})dt,
    \end{eqnarray*}
for all $T>0$. Using the Gronwall's inequality, we have
    \begin{eqnarray*}
     && \|w\|_{L^\infty([0,T];L^2)}
      +      \|w\|_{L^1([0,T];L^2)}\\
      &\leq& C e^{C\|(\nabla \mathbf{u},\nabla \mathbf{v}^\epsilon)\|_{L^1_TL^\infty}}(
      \|\mathcal{Q}\mathbf{v}^\epsilon\|_{L^1_TL^\infty}
      \|\nabla \mathbf{u}\|_{L^\infty_TL^2}+\|\mathbf{v}^\epsilon\|_{L^\infty_TL^\frac{2p}{p-2}}
      \|\nabla
\mathcal{Q}\mathbf{v}^\epsilon\|_{L^1_TL^p}),
    \end{eqnarray*}
 From
Theorem \ref{thm1.1}, we get
    $
    {\mathbf{v}^\epsilon}
    $ uniformly bounded in $
       L^\infty_t{B^{\sigma}_{2,1}}\cap L^1_t{B^{\sigma}_{2,1}} $.
From Proposition \ref{P-5.2}, we have
\begin{eqnarray*}
     && \|w\|_{L^\infty([0,T];L^2)}
      +      \|w\|_{L^1([0,T];L^2)}\\
      &\leq&\left\{\begin{array}{ll}
 C\epsilon^{\frac{(p-2)^2N}{8p(3p-4)}},
 & 1-\frac{(p-2)(N-4)}{2p}>0   \textrm{ and }2\leq p\leq N,\\
 C\epsilon^{\frac{(p-2)^2N^2}{8p^2(3p-4)}},
 & 1-\frac{(p-2)(N-4)}{2p}>0  \textrm{ and }N<p\leq\infty,\\
  C_\beta \epsilon^{\frac{p-2-\beta}{4p}}, &  1-\frac{(p-2)(N-4)}{2p}=0, \\
  C \epsilon^{\frac{p-2}{4p}}, &  1-\frac{(p-2)(N-4)}{2p}<0 \textrm{ and }2\leq p\leq N,\\
  C \epsilon^{\frac{N(p-2)}{4p^2}}, &  1-\frac{(p-2)(N-4)}{2p}<0 \textrm{ and }N<p\leq\infty.
 \end{array}
 \right.
 \end{eqnarray*}
 By interpolation, we can easily  obtain (\ref{M-E80}).
\end{proof}

\section*{Acknowledgements}
The research of Jiang Xu is partially supported by the NSFC
(11001127), China Postdoctoral Science Foundation (20110490134) and
Postdoctoral Science Foundation of Jiangsu Province (1102057C). The
research of T. Zhang is partially supported by the NSFC
(10931007, 10901137),  Zhejiang
Provincial Natural Science Foundation of China (Z6100217, Y12A010074), Program for New Century Excellent Talents in University NCET-11-0462,  and the Fundamental Research Funds for the Central Universities.

\bibliographystyle{elsarticle-num}

\end{document}